\documentclass[reqno]{amsart}
\usepackage{fullpage}
\usepackage{array}
\usepackage{hyperref}
\usepackage{graphicx,color}
\usepackage{float}
\usepackage{upgreek}

\newcommand{\gen}[1]{\ensuremath{\langle{#1}\rangle}}

\usepackage{comment}

\newcommand{\Gstar}{\ensuremath{(G,\ast)}-}

\newcommand{\IdGstar}[1]{\ensuremath{\textnormal{Id}^{(G,\ast)}({#1})}}

\newcommand{\cnGstar}{\ensuremath{c_n^{(G,\ast)}}}

\newtheorem{teorema}{Theorem}[section]
\newtheorem{exemplo}[teorema]{Example}

\newtheorem{lema}[teorema]{Lemma}
\newtheorem{corolario}[teorema]{Corollary}
\newtheorem{observacao}[teorema]{Remark}
\newtheorem{proposicao}[teorema]{Proposition}
\newtheorem{definicao}[teorema]{Definition}
\newcolumntype{C}[1]{>{\centering\let\newline\\\arraybackslash\hspace{0pt}}m{#1}}

\usepackage[utf8]{inputenc}
\usepackage{amsfonts}
\usepackage{amssymb}
\usepackage{verbatim}
\usepackage{mathrsfs}

\numberwithin{equation}{section}

\begin{document}
	\title{Minimal varieties of algebras with graded involution and quadratic growth}
	
\author[WESLEY Q. COTA AND ANA C. VIEIRA]{WESLEY Q. COTA AND ANA C. VIEIRA$^*$}
	
	\dedicatory{Departamento de Matemática, Instituto de Ciências Exatas, Universidade Federal de Minas Gerais. \\ Avenida Antonio Carlos 6627, 31123-970, Belo Horizonte, Brazil}

  \thanks{{\it E-mail addresses:} wesleyqc@ufmg.br (Cota) and acvufmg2011@gmail.com (Vieira).}

  	\thanks{\footnotesize {$^*$ Corresponding author}}

\thanks{The first author was partially supported by FAPEMIG and the second author was partially supported by FAPEMIG and by CNPq.}

	
	
	\begin{abstract} 
Subalgebras of upper triangular matrix algebras have played a fundamental role in the classification of minimal varieties of polynomial growth. Such classification has become a source of study in recent years since it leads to the more general classification of varieties of polynomial growth  $n^k$, as has already been proven in many contexts for several values of $k$.
 In this paper, we study the asymptotic behavior of the sequence of codimensions of algebras graded by a finite group $G$ and endowed with a graded involution $*$, also called $(G,*)$-algebras. 
We classify the minimal varieties generated by a finite-dimensional $(G,*)$-algebra with quadratic growth.
	\end{abstract}
	\maketitle
	
	\section{Introduction}
    The growth of polynomial identities for the algebra $M_n(F)$ of $n\times n$ matrices in characterstic zero was studied by Regev in \cite{RG1} and it has been also considered for any algebra in general. In particular, a finitely generated algebra $A$ over an infinite field $F$ satisfies all identities satisfied by $M_k(F)$, for some $k\geq 1$  \cite[Theorem 1.12.2]{GZlivro}. 
 
 In the last few years, classifications of varieties of algebras according to the behavior of the codimension sequence have been the subject of currently research.  
This sequence was introduced in 1972 by Regev \cite{RG} and for an algebra $A$ over a field of
characteristic zero, its $n$-th term is defined as $c_n(A)= \dim P_n(A)$, where $P_n(A)={P_n}/{P_n\cap \textnormal{Id}(A)}$. Here,  $P_n$ denotes the vector space of multilinear polynomials in the first $n$ non-commuting variables and $\textnormal{Id}(A)$ is the $T$-ideal of identities of $A$.

 Regev also proved that if $A$ is a PI-algebra over a field of characteristic zero, that is if $A$ satisfies a non-trivial identity, then its codimension sequence is exponentially bounded. This means that there exist constants $\beta , \alpha >0$ such that $c_n(A)\leq \beta\alpha ^{n}$, for all $n\geq 1$. In $1979$, Kemer \cite{Kem} established that $c_n(A)\leq qn^t$, for certain constants $q,t\geq 0$ and for all $n\geq 1$, if and only if the infinite-dimensional Grassmann algebra $\mathcal{G}=\langle 1,e_1,e_2, \ldots \mid e_ie_j=-e_je_i\rangle$ and the algebra $UT_2$ of $2\times 2$ upper triangular matrices do not belong to the variety generated by $A$. In the last case, we say that the algebra $A$, or the variety $\textnormal{var}(A)$ generated by $A$, has polynomial growth $n^k$ if $c_n(A)\approx \alpha n^k$, for some $\alpha >0$.

In recent years, many classifications of varieties of polynomial growth $n^k$ have been presented, according to the value of $k$. As examples we cite the classification of varieties of linear growth by Giambruno and La Mattina \cite{GLa} and of the varieties of polynomial growth $n^{k}$, for $k\leq 3$, generated by unitary algebras, by Giambruno, La Mattina and Petrogradsky \cite{Petro}. Moreover, de Oliveira and Vieira \cite{Mara3} presented a list of $16$ algebras generating all varieties generated by unitary algebras with polynomial growth $n^4$.  

The precedent classifications evidenced an important class of algebras, the so-called minimal varieties. We say that a variety $\mathcal{V}$ is minimal of polynomial growth $n^k$ if $c_n(\mathcal{V})\approx \alpha n^k$, for some $\alpha >0$, and for any proper subvariety $\mathcal{U}\subsetneq \mathcal{V}$ we have $c_n(\mathcal{U})\approx \beta n^p$ with $p<k$. In the previous results, the authors proved that the respective classifications consist essentially of varieties generated by a finite direct sum of algebras generating minimal varieties. Motivated by this remark, Giambruno, La Mattina and Zaicev \cite{Minimais} presented a characterization of minimal varieties of polynomial growth $n^k$ and proved that for $k\leq 5$ we have a finitely many minimal varieties of growth $n^k$ but for $k> 5$, we have infinitely many ones.

The previous concepts and results have been extended to the context of algebras with additional structure such as algebras with involution (also called $*$-algebras), superalgebras and algebras with graded involution. 

It is important to emphasize the constant presence of subalgebras of upper triangular matrix algebras in the given classifications in all situations cited above.

The characterization of minimal varieties generated by a unitary algebra of polynomial growth $n^t$ was given by Gouveia, dos Santos and Vieira \cite{tatiana} for the class of $*$-algebras and superalgebras and they proved that there exist finitely many such varieties only if $t\leq 2$ .

The classification of the varieties of $*$-algebras and also varieties of superalgebras  of linear codimension growth were presented respectively by La Mattina and Misso \cite{lamattinamisso} and Giambruno, La Mattina and Misso \cite{GLMisso}. Moreover, in \cite{Dafne} Bessades, dos Santos, Santos and Vieira classified the varieties generated by unitary $*$-algebras and superalgebras with quadratic growth.   We emphasize that, in all the respective situations, the classifications were obtained by the direct sum of algebras generating minimal varieties. 

In this work, we consider $G$ a finite group and $A$ a $*$-algebra graded by $G$. We say that $A$ is a $(G,*)$-algebra if for all  $g\in G$, the homogeneous component $A_g$ of the grading is invariant under the involution. Our goal is to study the behavior of the codimension sequence of $(G,*)$-varieties of polynomial growth.

In \cite{Mara}, de Oliveira, dos Santos and Vieira classified the varieties of $(G,*)$-algebras with at most linear growth and proved that they are generated by a finite direct sum of $(G,*)$-algebras generating minimal $(G,*)$-varieties. The main purpose of this paper is to classify the minimal $(G,*)$-varieties of quadratic growth generated by a finite-dimensional $(G,*)$-algebra, where $G$ is any finite group.

It is worth to mention that in the particular case $G\cong \mathbb{Z}_2$, Ioppolo, dos Santos, Santos and Vieira \cite{mallu} classified the minimal $(\mathbb{Z}_2,*)$-varieties of quadratic growth. Using such classification, Bessades, Costa and Santos \cite{Mallu} concluded the general classification of $(\mathbb{Z}_2,*)$-varieties generated by unitary algebras of quadratic growth. In addition to extending the results in \cite{mallu}, the classification of minimal varieties given here will be useful for presenting the complete classification of $(G,*)$-varieties of quadratic growth in the future.

\section{\texorpdfstring{ Generalities about $(G,*)$-algebras}{Generalities about (G,*)-algebras}}

In all this paper, $A$ denotes an associative algebra over a field $F$ of characteristic zero and $G$ denotes a multiplicative group with unit $1$.
In this section we introduce our main object of study, the so-called $(G,*)$-algebras, we define the $(G,*)$-codimension sequence and recall some results established in the last years regarding the asymptotic behavior of this sequence. 

Recall that a linear map $*:A\rightarrow A$ defined on an algebra $A$ is called an involution if $(ab)^{*}=b^*a^*$ and $(a^*)^*=a$, for all $a, b \in A.$ An algebra $A$ endowed with an involution $*$ is called an algebra with involution or a $*$-algebra. It is known that if $A$ is a $*$-algebra, then we can write $A=A^+\oplus A^-$, where $A^+ = \{a\in A\mid a^*=a\}$ is the subspace of symmetric elements and $A^-= \{a\in A\mid a^*=-a\}$ is the subspace of skew elements. 

For instance, any commutative algebra is a $*$-algebra endowed with the trivial involution $a^*=a$, for all $a\in A$. Moreover, for the subalgebra $M:=F(e_{11}+e_{44})+F(e_{22}+e_{33})+ Fe_{12}+Fe_{34}$ of the $4\times 4$ upper triangular matrices we consider the reflection involution obtained by reflecting a matrix along its secondary diagonal, i.e.
		$$\left( {\begin{array}{*{20}c}
				a & b & 0 & 0  \\
				0 & c & 0 & 0\\
				0 & 0 & c & d  \\
				0 & 0 & 0 & a \\
		\end{array}} \right)^{*}= \left( {\begin{array}{*{20}c}
				a & d & 0 & 0  \\
				0 & c & 0 & 0\\
				0 & 0 & c & b  \\
				0 & 0 & 0 & a \\
		\end{array}} \right).$$

We can define $\textnormal{Id}^*(A)$ the $T_*$-ideal of $A$, consisting of all $*$-identities of $A$ and define $c_n^*(A)$ the $n$-th $*$-codimension of $A$ as an extension of the notion of codimension given in the ordinary case.


We say that $A$ is a $G$-graded algebra if there exist subspaces $A_{g},\, g\in G$, called homogeneous components of degree $g$, satisfying $A=\underset{g\in G}{\bigoplus}A_{g}$ and $A_{g}A_{h}\subseteq A_{gh}$ for all $g,h\in G$. The support of a $G$-graded algebra $A$ as the set $supp(A)=\{g\in G\mid A_g\neq \{0\}\}$.

Notice that any algebra is a $G$-graded algebra endowed with the trivial $G$-grading given by $A_1=A$ and $A_g=\{0\}$, for all $g\in G- \{1\}$. A $G$-grading can be defined on $UT_m$, the algebra of $m\times m$ upper triangular matrices, by considering a $m$-tuple $\mathsf{h}=(h_1, \ldots, h_m) \in G^m$ of elements of $G$ and defining $(UT_m)_{g}= \mbox{span}_F\{e_{ij}\mid h_i^{-1}h_j=g\}$, for all $g\in G.$ This $G$-grading is called an elementary $G$-grading induced by $\mathsf{h}$.

For $g\in G$ with $|g|=p$ a prime number, we denote by $C_p=\langle g \rangle $ the cyclic subgroup of $G$ of order $p$  and $FC_p$ to be the group algebra of $C_p$ over $F$. We consider a $G$-grading on $FC_p$ induced by $C_p$, given by $(FC_p)_{g^i}=F g^i$, for all $0\leq i \leq p-1$, and $(FC_p)_{h}=\{0\}$ for all $h\in G- C_p$.



Similarly, we consider $\textnormal{Id}^G(A)$ the $T_G$-ideal of $A$ and define $c_n^G(A)$ the $n$-th $G$-graded codimension of $A$, as done for the previous structures.

An involution $*$ defined on a $G$-graded algebra $A$ is called a graded involution if it preserves the homogeneous components of $A$, i.e  $A_g^*= A_g$, for all $g\in G$. 
Finally, we can define our main object of study.
 
 \begin{definicao}
     A $G$-graded algebra $A$ endowed with a graded involution $*$ is called a \Gstar algebra. 
 \end{definicao}

It is quite immediately to see that if a $G$-graded algebra $A$ admits a graded involution then $supp(A)$ is a subset of $Z(G)$, the center of $G$. Therefore, from now on, we will assume that $G$ as a finite abelian group of order $t$. In particular, if $G\cong \mathbb{Z}_2$ is the cyclic group of order $2$ then the $(\mathbb{Z}_2,*)$-algebra is known as a $*$-superalgebra. 
 
  For instance, consider the following \Gstar algebras:
\begin{enumerate}

\item[1.] Any $*$-algebra is a $(G,*)$-algebra endowed with the trivial $G$-grading. In particular, we denote by $M_{1,\rho}$ the algebra $M$ with trivial $G$-grading and reflection involution.

 \item[3.] A commutative $G$-graded algebra is a $(G,*)$-algebra with trivial involution. For instance, we consider $FC_p^{G}$ the group algebra $FC_p$ with $G$-grading induced by $C_p$ and trivial involution.

\item[4.] For a group $G$ of even order, we consider $g\in G$ with $|g|=2$ and $C_2=\langle g\rangle $. Denote by $(FC_2)^{(G,*)}$ the group algebra $FC_2$ with $G$-grading induced by $C_2$ and involution given by $(a_1+a_2g)^{*}= a_1-a_2g$. Also, we denote by $(FC_2)_*$ the group algebra $FC_2$ with the trivial $G$-grading and the involution defined before.

\item[5.] Let $g\in G- \{1\}$ and denote by $M_{g,\rho}$ the \Gstar algebra $M$ endowed with reflection involution and $G$-grading given by
		$$(M_{g,\rho})_{1}=F(e_{11}+e_{44})+F(e_{22}+e_{33}),$$ $$(M_{g,\rho})_{g} =F e_{12}+F e_{34} \mbox{ and }(M_{g,\rho})_{h} = \{0\}, \mbox{ for all }h\in G- \{1,g\}.$$
\end{enumerate}

Since the involution $*$ is a graded then $A$ can be decomposed into a direct sum of vector subspaces 
	$$A=\underset{g\in G}{\bigoplus}(A_{g}^+ +  A_{g}^-),$$ where $A_{g}^+= \{a\in A_g \mid a^*=a\}$ and $A_{g}^-= \{a\in A_g \mid a^*=-a\}$ are, respectively, the symmetric component and the skew component of homogeneous degree $g$.
  
  For all $g\in G$, consider $Y=\underset{g\in G}{\bigcup }Y_g$ and $Z=\underset{g\in G}{\bigcup }Z_g$, where $Y_g= \{y_{1,g},y_{2,g},\ldots \}$ is the set of symmetric variables of homogeneous degree $g$ and $Z_g= \{z_{1,g},z_{2,g},\ldots\}$ is the set of skew variables of homogeneous degree $g$. Let $\mathcal{F}=F\langle Y\cup Z\mid G \rangle$ be the free associative \Gstar algebra generated by $Y\cup Z$ over $F$, whose elements are called \Gstar polynomials. 
  
\begin{definicao}
		A \Gstar polynomial $f=f(y_{1,1},\ldots, y_{i_1,1}, z_{1,1},\ldots ,z_{j_1,1},\ldots$, $y_{1,g_{t}}, \ldots ,$ $y_{i_{t},g_{t}},$ $z_{1,g_{t}}$, $\ldots , z_{j_{t},g_{t}})$ $\in \mathcal{F}$ is a \Gstar identity of a $(G,*)$-algebra $A$ if 
		$$f(a_{1,1}^+,\ldots, a_{i_1,1}^+, a_{1,1}^-,\ldots ,a_{j_1,1}^-, \ldots,a_{1,g_{t}}^+,\ldots ,a_{i_{t},g_{t}}^+,a_{1,g_{t}}^-,\ldots ,a_{j_{t},g_{t}}^-)=0 ,$$
		for all $a_{1,1}^+,\ldots, a_{i_1,1}^+ \in A_{1}^{+}$,
		$a_{1,1}^-,\ldots ,a_{j_1,1}^- \in A_{1}^{-}$, $\ldots$, $a_{1,g_{t}}^+,\ldots ,a_{i_{t},g_{t}}^+ \in A_{g_t}^{+}$, $a_{1,g_{t}}^-,\ldots , a_{j_{t},g_{t}}^- \in A_{g_t}^{-}$.
	\end{definicao}

Let $\textnormal{Id}^{(G,*)}(A)$ be the set of all \Gstar identities of $A$, which is called the $T_{(G,*)}$-ideal of $A$. Note that $\textnormal{Id}^{(G,*)}(A)$ is an ideal invariant under all endomorphisms of ${F}\langle Y\cup Z\mid G \rangle$ that preserve the grading and commute with the involution. Since $char(F)= 0$ then the $T_{(G,*)}$-ideal of $A$ is generated by its multilinear \Gstar identities and so we consider $P_n^{(G,*)}=\mbox{span}_F\{w_{\sigma(1)} \cdots$ $w_{\sigma(n)}$$\mid \sigma\in S_n,\; w_i\in \{y_{i,g},z_{i,g}\},\; g\in G \}$ the space of  multilinear \Gstar polynomials of degree $n$.

In all this paper, for $x_{i,r}$ we mean a variable in the set $\{y_{i,r,}z_{i,r}\},$ for some $r\in G$.
	
	\begin{exemplo} \cite{Mara} \cite{Valenti} \label{43} For $g\in G- \{1\}$ and for all $h\in G- \{1\}$ and $r\in G- \{1,g\}$ we have that 
		\begin{enumerate} 
		    \item[1)]	$\textnormal{Id}^{(G,*)}{((FC_2)_*)}=\langle [x_{1,1},x_{2,1}], x_{1,h}\rangle_{T_{(G,*)}}.$
	\item[2)] 	$\textnormal{Id}^{(G,*)}(M_{g,\rho})=\langle z_{1,1}, x_{1,g}x_{2,g}, x_{1,r}\rangle_{T_{(G,*)}}$ and $\textnormal{Id}^{(G,*)}{(M_{1,\rho})}=\langle z_{1,1}z_{2,1}, x_{1,h} \rangle_{T_{(G,*)}}.$
			\item[3)] If $g\in G$ is an element of prime order $p$, $C_p= \gen{g}$ and $s\in G- C_p$ then $$\textnormal{Id}^{G}(FC_p^G)= \langle [y_{1,g^i},y_{2,g^j}], z_{1,g^i}, x_{1,s} \mid 0\leq i\leq j\leq p-1  \rangle_{T_{G}}.$$       
			\item[4)] If $|G|$ is even and $g\in G$ with $|g|=2$ generating a cyclic subgroup $C_2$ of $G$ then 
		$$\textnormal{Id}^{(G,*)}{((FC_2)^{(G,*)})}= \langle z_{1,1}, y_{1,g}, x_{1,r}  \rangle_{T_{(G, *)}}.$$
			
		\end{enumerate}
	\end{exemplo}

The $(G,*)$-variety $\mathcal{V}=\textnormal{var}^{(G,*)}(A)$ generated by $A$ is the class of all $(G,*)$-algebras $B$ such that $\textnormal{Id}^{(G,*)}(A)\subseteq \textnormal{Id}^{(G,*)}(B)$ and we say that $A$ and $B$ are $T_{(G,*)}$-equivalent if $\textnormal{Id}^{(G,*)}(A)=\textnormal{Id}^{(G,*)}(B)$. 

For $n\geq 1$, the $n$-th $(G,*)$-codimension of a $(G,*)$-algebra $A$ is defined by $$c_n^{(G,*)}(A):=\dim_F \dfrac{P_n ^{(G,*)}}{P_n^{(G,*)}\cap \textnormal{Id}^{(G,*)}(A)}\cdot$$ 

Moreover, if $\mathcal{V} = \textnormal{var}^{(G,*)}(A)$ then we define $\mbox{Id}^{(G,*)}(\mathcal{V}) = \mbox{Id}^{(G,*)}(A)$ and $c_n^{(G,*)}(\mathcal{V}) = c_n^{(G,*)}(A)$. 

The relationship between the \Gstar codimension, $*$-codimension and $G$-graded codimension of a \Gstar algebra $A$ was given by Oliveira, dos Santos and Vieira in \cite{Lorena}.
	
	\begin{lema} \label{codimensoes} For a \Gstar algebra $A$ we have 
		 $$c_n(A)\leq c_n^{*}(A) \leq \cnGstar(A),$$ $$ c_n(A)\leq c_n^{G}(A) \leq \cnGstar(A)\mbox{ and }c_n(A)\leq \cnGstar(A)\leq 2^n|G|^nc_n(A).$$
		
	\end{lema}

 Consequently, if a $(G,*)$-algebra $A$ satisfy an ordinary non-trivial identity then, by \cite{RG}, the $(G,*)$-codimension sequence $c_n^{(G,*)}(A)$, $n\geq 1$, is also exponentially bounded. This result motivates the following definition.
 
 \begin{definicao}
A $(G,*)$-algebra has polynomial growth if $c_n^{(G,*)}(A)\leq \alpha n^t$, for some $\alpha, t \geq 0$ and for all $n\geq 1$.    Also we say that $A$ has polynomial growth $n^k$ if $c_n^{(G,*)}(A)\approx \alpha n^k$, where $\alpha >0$. 
 \end{definicao}

Recall that $G=\{g_1=1, \ldots , g_t\}$ and consider $n=n_1+\cdots + n_{2t}$ a sum of $2t$ non-negative integers and denote $\langle n \rangle=(n_1, \ldots , n_{2t})$. Let $P_{\langle n \rangle}$ be the subspace of $P_n^{(G,*)}$ where the first $n_1$ variables are symmetric of homogeneous degree $g_1$, the second $n_2$ variables are skew of homogeneous degree $g_1$, the third $n_3$ variables are symmetric of homogeneous degree $g_2$, the fourth $n_4$ variables are skew of homogeneous degree $g_2$ and so on. Notice that there are $\displaystyle\binom{n}{\langle n \rangle}:= \displaystyle\binom{n}{n_1, \ldots , n_{2t}}$ subspaces isomorphic to $P_{\langle n \rangle}$ in $P_n^{(G,*)}$. Moreover, it is easy to check that $$P_n^{(G,*)} \cong \displaystyle \bigoplus_{\langle n \rangle} \displaystyle\binom{n}{\langle n \rangle} P_{\langle n \rangle}.$$

Define the following space $$P_{\langle n \rangle}(A)= \frac{P_{\langle n \rangle}}{P_{\langle n \rangle} \cap \textnormal{Id}^{(G,*)}(A)} $$ and consider $c_{\langle n \rangle}(A)= \dim_F P_{\langle n \rangle}(A)$ the $\langle n \rangle$-codimension of a $(G,*)$-algebra $A$. The relationship between $c_n^{(G,*)}(A)$ and the $\langle n \rangle$-codimensions of $A$ is given by the following equation 
	\begin{equation} \label{293}
		c_n^{(G,*)}(A)= \underset{\langle n \rangle  }{\sum} \displaystyle\binom{n}{\langle n \rangle } c_{\langle n \rangle}(A). 
	\end{equation}
	
For a sum $n=n_1+ \cdots + n_{2t}$ we notice that there is a natural left action of the product of symmetric groups $S_{\langle n \rangle}:= S_{n_1} \times \cdots \times S_{n_{2t}}$ on $P_{\langle n \rangle}(A)$ and so it is an $S_{\langle n \rangle}$-module.  It is known that the irreducibles $S_{\langle n \rangle}$-characters are outer tensor product of irreducible $S_{n_i}$-characters and there exists a one-to-one correspondence between partitions $\lambda_i=((\lambda_i)_1,(\lambda_i)_2, \ldots , (\lambda_i)_{k_i})$ of $n_i$ and the irreducible $S_{n_i}$-characters. Thus, by complete reducibility we may consider 
	\begin{equation} \label{character}
		\chi_{\langle n \rangle }(A)=\underset{\langle \lambda \rangle \vdash \langle n \rangle}{\sum} {m}_{\langle \lambda \rangle} \chi_{\lambda_1} \otimes \cdots \otimes \chi_{\lambda_{2t}}
	\end{equation} the decomposition of the $\langle n \rangle$-character of the space $P_{\langle n \rangle}(A)$ into irreducible $S_{\langle n \rangle}$-characters, which is called $\langle n \rangle$-cocharacter of $A$, where $\chi_{\lambda_1} \otimes \cdots \otimes \chi_{\lambda_{2t}}$ is the  the irreducible $S_{\langle n \rangle}$-character associated to a multipartition $\langle \lambda \rangle := ( \lambda_1, \ldots , \lambda_{2t})  \vdash \langle n \rangle$ and ${m}_{\langle \lambda \rangle }$ denotes the corresponding multiplicity. The degree of  the irreducible $S_{\langle n \rangle}$-character $\chi_{\lambda_1} \otimes \cdots \otimes \chi_{\lambda_{2t}}$ is given by $d_{\lambda_1} \cdots d_{\lambda _{2t}},$ where $d_{\lambda_i}$ is the degree of the $S_{n_i}$-character associated to $\lambda_i$ calculated by the hook formula \cite[Theorem 3.10.2]{Sagan}.

 For all possible sums $n=n_1+\cdots + n_{2t}$ we consider the $\gen{n}$-cocharacters of $A$ as given in (\ref{character}) in order to define the \Gstar colength of $A$ as
		$$l_n^{(G, *)}(A)=\displaystyle \sum_{\langle n\rangle }\sum\limits_{\displaystyle{
				\scriptsize{\begin{array}{c}
						\gen{\lambda}\vdash \gen{n}
		\end{array}}}}{m}_{\gen{\lambda}}.$$

There is a well established method to calculate the multiplicities $m_{\gen{\lambda}}$ in the decomposition (\ref{character}) of the $\gen{n}$-cocharacter. However, we will give a brief description, omit the details and recommend the reference \cite{Drensky} (see Section $12.4$) and \cite{WesAnaRaf} for more information.

For $m\geq 1$, we define $X^m= \underset{g\in G}{\bigcup } (Y_g^m\cup Z_g^m)$, where $Y_g^{m}= \{y_{1,g}, \ldots , y_{m,g}\}$ and $Z_g^{m}= \{z_{1,g}, \ldots , z_{m,g}\}$. Denote by $F_m:= F\langle X^m\mid G\rangle $ the free associative $(G,*)$-algebra generated by $X^m$ over $F$. Define $F_m^n$ the subspace of the homogeneous polynomials in $F_m$ with degree $n\geq m$ and notice that the group $GL_m^{2t}:=GL_m \times \cdots \times GL_m$, the direct product of $2t$-copies of $GL_m$, acts diagonally on $F_m^n$. Since $F_m^n \cap \textnormal{Id}^{(G,*)}(A)$ is invariant under this action then the quotient space
	$$F_m^n(A)=\frac{F_m^n}{F_m^n \cap \textnormal{Id}^{(G,*)}(A)}$$ has a structure of $GL_m^{2t}$-module. By \cite[Theorem 12.4.12]{Drensky}, each irreducible $GL_m^{2t}$-module is cyclic and it is generated by a non-zero polynomial $f_{\langle \lambda \rangle}$ called the highest weight vector associated to the multipartition $\langle \lambda \rangle$ of $\gen{n}$. For each multipartition $\langle \lambda \rangle \vdash \langle n \rangle$ we consider a multitableau $T_{\langle \lambda \rangle}:=(T_{\lambda_1},\ldots , T_{\lambda_{2t}})$ consisting of $2t$ Young tableaux and denote by $f_{T_{\langle \lambda \rangle}}$ the highest weight vector associated to $T_{\langle \lambda \rangle}$ defined in \cite{WesAnaRaf}. By \cite[Theorem 12.4.12]{Drensky}, the $(G,*)$-polynomial $f_{\langle \lambda \rangle}$ can be written uniquely as a linear combination of the polynomials $f_{T_{\langle \lambda \rangle}}$. Moreover, the multiplicity $m_{\langle \lambda \rangle}$ is equal to the maximal number of linearly independent highest weight vectors $f_{T_{\langle \lambda \rangle}}$ in $F_{m}^n(A)$ (see for instance \cite{WesAnaRaf}).

It will be convenient to use the notation $\langle \lambda \rangle = ((\lambda_{i_1})_{{g_{i_1}^+}}, (\lambda_{i_2})_{g_{i_2}^-}, \ldots  )$, where $(\lambda_{i_1})_{{g_{i_1}^+}}$ means that $(\lambda_{i_1})$ is the partition of $n_{2i_1-1}$, $(\lambda_{i_2})_{g_{i_2}^-}$ means that $(\lambda_{i_2})$ is a partition of $n_{2i_2}$ and so on. The empty partitions are omitted in this notation. Also, we consider the notation $T_{\gen{n}}=({T_{(\lambda_{i_1})}}_{{g_{i_1}^+}}, {T_{(\lambda_{i_2})}}_{g_{i_2}^-}, \ldots  )$ to indicate the respective tableaux associated to $(\lambda_{i_1})$, $(\lambda_{i_2})$ and so on.

Many characterizations of polynomial growth for $(G,*)$-varieties have been presented in the last years. Here we collect the information presented in \cite[Theorem 5.8]{Lorena}, \cite[Theorem $4.5$]{Mara} and \cite{Mara2} about the growth of polynomial identities satisfied by $A$. In the following, by $J=J(A)$ we denote the Jacobson radical of $A$.
	
	\begin{teorema}  \label{polynomialgrowth}
		Let $\mathcal{V}=\textnormal{var}^{(G,*)}(A)$ be a $(G,*)$-variety generated by a finite-dimensional \Gstar algebra $A$ over a field $F$ of characteristic zero. The following conditions are equivalent:

  \begin{enumerate}
      \item[1)] $A$ has polynomial growth.

      \item[2)] Either $|G|$ is odd and $(FC_2)_{*}$, $FC_p^G$, $M_{g,\rho}$ $\notin \mathcal{V}$,  for all prime $p$ dividing $|G|$ and for all $g \in G$ or $|G|$ is even and $(FC_2)_{*}$, $(FC_2)^{(G,*)}$, $FC_p^G$, $M_{g,\rho}$ $\notin \mathcal{V}$  for all prime $p$ dividing $|G|$ and for all $g \in G$.

      \item[3)] $A\sim_{T_{(G, *)}} B_1\oplus\cdots\oplus B_m$, where $B_1, \ldots, B_m$ are finite-dimensional \Gstar algebras over $F$ with $\dim_F {B_i}/{J(B_i)}\leq 1$, for all $i=1, \ldots, m$.

      \item[4)] For all $n=n_1+\cdots + n_{2t}$ and $\langle \lambda \rangle \vdash \gen{n}$ we have
      $$ \chi_{\langle n \rangle }(A)=\underset{ \tiny{\begin{array}{c}
					\gen{\lambda}\vdash \gen{n} \\
					n-(\lambda_1)_1 < q
		\end{array}}}{\sum} {m}_{\langle \lambda \rangle} \chi_{\lambda_1} \otimes \cdots \otimes \chi_{\lambda_{2t}},$$ where $q>0$ is such that $J^{q-1}\neq \{0\}$ and $J^{q}=\{0\}.$

  \item[5)] There exists a constant $h\geq 0$ such that $l_n^{(G,*)}(A)\leq h$, for all $n\geq 1$.
  \end{enumerate}

	\end{teorema}

   The previous results have paved the way for significant progress in understanding these varieties, particularly in identifying how the growth of the identities of $A$ is influenced by its algebraic properties. The next proposition reveals a connection with the nilpotency index of the Jacobson radical and its polynomial growth. Before presenting the proposition, we introduce a necessary technical lemma. In the following, for a non-negative integer $r$ we denote by $\#r$ the number of partitions of $r$. 

\begin{lema} \label{cotas}
 Let $q> 0$ a constant and consider the sets $$\Delta_n= \{(n_1, \ldots ,n_{2t})\mid n=n_1+\cdots + n_{2t}, n_1 > n-q\}$$  $$\mathcal{T}_{n_1, \ldots , n_{2t}}= \{(\lambda_1, \ldots , \lambda_{2t})\mid \lambda_i\vdash n_i,  (\lambda_1)_1 > n-q\}$$ where $n=n_1+\cdots + n_{2t}$. The cardinalities $|\Delta_n|$ and $|\mathcal{T}_{n_1, \ldots , n_{2t}}|$ are bounded by constants, which do not depend on $n$ and on the integers $n_1, \ldots , n_{2t}$.
\end{lema}

\begin{proof}
 A combinatory argument proves that $|\Delta_n|=\displaystyle\binom{(q-1)+2t-1}{2t-1}$, which is a constant. For the set $\mathcal{T}_{n_1, \ldots , n_{2t}}$, we start noticing that $n_j < q$, for all $2\leq j\leq 2t$, and so $\#n_j< \#q$. Moreover, $\lambda_1 - (\lambda_1)_1
\leq n - (\lambda_1)_1  < q$ and so
for each one of the $q-1$ choices of $(\lambda_1)_1$ we have that the number of partitions $((\lambda_1)_2, (\lambda_1)_3, \ldots , (\lambda_1)_{k_1} ) \vdash \lambda_1-(\lambda_1)_1$ is bounded by $\#q$. Therefore, the number of partitions $\lambda_1$ of $n_1$ satisfying $(\lambda_1)_1 > n-q$ is less than $(q-1)\#q$. Therefore $|\mathcal{T}_{n_1, \ldots , n_{2t}}|\leq (q-1)(\#q)^{2t}$, which is a constant that does not depends on $n$ and on the integers $n_1, \ldots , n_{2t}$ such that $n_1+\cdots + n_{2t}=n$.
\end{proof}

\begin{proposicao} \label{prop1}
    Let $A$ be a finite-dimensional $(G,*)$-algebra of polynomial growth. Then $c_n^{(G,*)}(A)\leq \delta n^{q-1}$ where $q$ is the nilpotency index of the Jacobson radical of $A$ and $\delta > 0$ is a constant.
\end{proposicao}

\begin{proof}

Since $A$ is a finite-dimensional $(G,*)$-algebra of polynomial growth, by Theorem \ref{polynomialgrowth}, for all $n=n_1+\cdots + n_{2t}$ and $\langle \lambda \rangle \vdash \gen{n}$ we have
     $$ \chi_{\langle n \rangle }(A)=\underset{ \tiny{\begin{array}{c}
					\gen{\lambda}\vdash \gen{n} \\
					n-(\lambda_1)_1 < q
		\end{array}}}{\sum} {m}_{\langle \lambda \rangle} \chi_{\lambda_1} \otimes \cdots \otimes \chi_{\lambda_{2t}},$$ where $q$ is such that $J^{q}=\{0\}.$ By Theorem \ref{polynomialgrowth}, there exists a constant $h\geq 0$ such that $l_n^{(G,*)}\leq h$, for all $n\geq 1$, and so $m_{\gen{\lambda}}\leq h$, for all $\gen{\lambda}\vdash \gen{n}$. Notice that if $r:= n_1-(\lambda_1)_1$ is the number of box below the first row of the Young diagram associated to $\lambda_1$ then, by the hook formula, we have $$d_{\lambda_1}\leq \frac{n_1!}{(n_1-r)}\leq cn_1^r\leq pn^{r},\mbox{ for some }p,c\geq 0.$$ Moreover, for all $2\leq i\leq 2t$ we have $n_i< q$ and so $d_{\lambda_i}\leq \beta$, where $\beta = \displaystyle \underset{2\leq j\leq 2t,\, \lambda_j \vdash n_j}{\mbox{max}} d_{\lambda_j}$ is a constant. Therefore, by Lemma \ref{cotas}, for all $n=n_1+\cdots + n_{2t}$ we have $$c_{\gen{n}}(A)= \underset{ \tiny{\begin{array}{c}
					\gen{\lambda}\vdash \gen{n} \\
					n-(\lambda_1)_1 < q
		\end{array}}}{\sum} {m}_{\langle \lambda \rangle} d_{\lambda_1} \cdots  d_{\lambda_{2t}} \leq \underset{ \tiny{\begin{array}{c}
					\gen{\lambda}\vdash \gen{n} \\
					n-(\lambda_1)_1 < q
		\end{array}}}{\sum} h d_{\lambda_1} \cdots  d_{\lambda_{2t}}\leq h\beta^{2t-1}|\mathcal{T}_{n_1, \ldots ,n_{2t}}|pn^r\leq  \alpha n^r ,$$ where $\alpha =h\beta^{2t-1}p(q-1)(\#q)^{2t}.$ Define $s=n_2+ \cdots + n_{2t}$ and observe that $$\binom{n}{\gen{n}}=\frac{n!}{n_1!\cdots n_{2t}!}\leq \frac{n!}{n_1!}=\frac{n!}{(n-s)!}\leq \beta n^s.$$

Finally, collecting all information and using equation (\ref{293}), we have $$c_n^{(G,*)}(A)\leq \underset{\tiny{\begin{array}{c}
					 \langle n \rangle \\
					n_1>n- q
		\end{array}} }{\sum} \displaystyle\binom{n}{\langle n \rangle } \alpha n^{r}\leq
        \underset{\tiny{\begin{array}{c}
					 \langle n \rangle \\
					n_1>n- q
		\end{array}} }{\sum} \displaystyle \alpha \beta  n^{r+s} \leq   \underset{\tiny{\begin{array}{c}
					 \langle n \rangle \\
					n_1>n- q
		\end{array}} }{\sum} \displaystyle \alpha \beta  n^{n-(\lambda_1)_1}\leq 
        \delta |\Delta_n| n^{q-1}.$$
  
\end{proof} 
  
 \section{\texorpdfstring{$(G,*)$-algebras of quadratic growth}{(G,*)-algebras of quadratic growth}}\label{Sec3}

 This section is dedicated to introducing some examples of
$(G,*)$-algebras. In particular, we are interested in those ones with quadratic growth of the codimension sequence.

 For $m\geq 2$, consider $E = \sum\limits_{i = 1}^{m-1} e_{i,i+1} + e_{2m-i,2m-i+1} \in UT_{2m}$ and define the following subalgebras of $UT_{2m}$
	\begin{align*}
		A_m &= \mbox{span}_F\{e_{11}+e_{2m,2m},E, \ldots , E^{m-2}; e_{12}, e_{13}, \ldots , e_{1m}, e_{m+1,2m}, e_{m+2,2m}, \ldots , e_{2m-1,2m}\},
		\\N_m &= \mbox{span}_F \{I_{2m}, E, \ldots , E^{m-2}; e_{12} - e_{2m-1,2m}, e_{13}, \ldots , e_{1m}, e_{m+1,2m}, e_{m+2,2m}, \ldots , e_{2m-2,2m}\},
		\\ U_m &= \mbox{span}_F \{ I_{2m}, E, \ldots , E^{m-2}; e_{12} + e_{2m-1,2m}, e_{13}, \ldots , e_{1m}, e_{m+1,2m}, e_{m+2,2m}, \ldots , e_{2m-2,2m}\}.
	\end{align*}

For $m\geq 2$ and $g\in G- \{1\}$ we define the following $(G,*)$-algebras: \vspace{0.3cm}

\begin{enumerate}
\item[$\bullet$] $N_{m}^*, U_{m}^{*}$ and $A_{m}^*$ the algebras $N_m, U_m$ and $A_m$, respectively, with trivial $G$-grading and reflection involution.
    \item[$\bullet$] $N_{m}^\mathsf{g}, U_{m}^\mathsf{g}$ and $A_{m}^\mathsf{g}$ are the algebras $N_m, U_m$ and $A_m$, respectively, endowed with reflection involution and $G$-grading induced by $\mathsf{g}=(1,  g^{m-1}, 1^{m-1}, g)\in G^{2m}$.
\end{enumerate}\vspace{0.3cm}

The previous $(G,*)$-algebras were considered by de Oliveira, dos Santos and Vieira \cite{Mara} in the classification of the subvarieties of $M_{g, \rho}$ as we can see below.

\begin{teorema}\cite[Theorem $5.15$]{Mara} \label{31} 
		Let $g \in G- \{1\}$. If $A \in \textnormal{var}^{(G,*)}(M_{g, \rho})$ then $A$ is $T_{(G, *)}$-equivalent to one of the following \Gstar algebras: $M_{g, \rho}$, $N$, $C\oplus N$, $N_{m}^{\mathsf{g}} \oplus N$, $U_{m}^{\mathsf{g}} \oplus N$, $N_{m}^{\mathsf{g}} \oplus U_{m}^{\mathsf{g}} \oplus N$, $A_{t}^{\mathsf{g}} \oplus N$, $N_{m}^{\mathsf{g}} \oplus A_{t}^{\mathsf{g}} \oplus N$, $U_{m}^{\mathsf{g}} \oplus A_{t}^{\mathsf{g}} \oplus N$, $N_{m}^{\mathsf{g}} \oplus U_{m}^{\mathsf{g}} \oplus A_{t}^{\mathsf{g}} \oplus N$, for some $m,t \geq 2$,  where $N$ is a nilpotent \Gstar algebra and $C$ is a commutative \Gstar algebra with trivial $G$-grading and trivial involution. 
	\end{teorema}

 The following result follows from \cite[Lemma 1]{LamaMartino}, \cite[Lemmas 3.8 and 3.10]{MN} and \cite[Lemmas 2 and 3]{LamaMartino}. In such papers is also possible to find the description of the ideals of $(G,*)$-identities of the cited algebras.

\begin{lema}
	\label{teoAk*}
	Let $m \ge 3$. Then we have 
$\cnGstar(A_{2}^{\ast}) =4n-1$, $\cnGstar(N_{2}^{\ast}) = n+1$, $\cnGstar(U_{2}^{\ast}) = 1$ and for all $A\in \{A_m^*,N_m^*, U_m^*\}$ we have $c_n^{(G,*)}(A)\approx \alpha n^{m-1}$, for some $\alpha\in F$ depending on $A$.
\end{lema}

In \cite{Mara}, de Oliveira, dos Santos and Vieira gave the description of $\IdGstar{N_{m}^{\mathsf{g}}}$, $\IdGstar{U_{m}^{\mathsf{g}}}$ and $\IdGstar{A_{m}^{\mathsf{g}}}$ and also proved the following lemma. 

\begin{lema} \label{NinMg*}
	Let $m\geq 2$ and $g \in G- \{1\}$.
	For all $A\in \{N_{m}^{\mathsf{g}},U_{m}^{\mathsf{g}}, A_{m}^{\mathsf{g}} \}$ we have $c_n^{(G,*)}(A)\approx \alpha n^{m-1}$, for some $\alpha\in F$ depending on $A$.
\end{lema}

For $m \geq 2$, we denote by $I_{m}$ the $m\times m$ identity matrix and let $E_1 = \sum\limits_{i = 1}^{m-1} e_{i,i+1} \in UT_{m}$. Define the commutative subalgebra of $UT_m$ given by
	$$C_{m}= \{\alpha I_m + \underset{1\leq i<m}{\sum} \alpha_i E_1^i\mid \alpha, \alpha_i \in F\}$$
 and consider the involution $*$ on $C_m$ defined by
 
 \begin{equation}\label{invCk}
     (\alpha I_{m} + \underset{1\leq i<m}{\sum} \alpha_i E_1^i)^*= \alpha I_{m} + \underset{1\leq i<m}{\sum} (-1)^i\alpha_i E_1^i.
 \end{equation}
	
	For each $ g \in G$ we consider on this algebra the only $G$-grading such that $$I_m\in (C_m)_1 \mbox{ and } E_1\in (C_m)_g.$$ Define the following $(G,*)$-algebras
 
		\begin{enumerate}
			\item[$\bullet$] $C_{m}^g$ is the algebra $C_m$ endowed with trivial involution and the $G$-grading defined as above.
			\item[$\bullet$] $C_{m,*}^g$ is the algebra $C_m$ endowed with the involution given in (\ref{invCk}) and $G$-grading defined as above.
		\end{enumerate}

Information about the ideals and $(G,*)$-codimensions of the previous $(G,*)$-algebras can be found in \cite{WesAnaRaf}, \cite{Lamattina} and \cite{Mara}. Here we present only the description of the $(G,*)$-codimensions.

\begin{lema} 
    Consider $g\in G$ with $|g|>2$ and if $|G|$ is even we consider $h\in G$ with $|h|=2.$ Then, 

\begin{enumerate}
    \item[1)]  $c_n^{(G,*)}(C_{3}^g)= c_n^{(G,*)}(C_{3,*}^g)=1+2n+\displaystyle\binom{n}{2}$.
    \item[2)] $\displaystyle \cnGstar({C}_{m,*}^{{1}})=\cnGstar({C}_{m}^{h})=\cnGstar({C}_{m,*}^{h})=\sum\limits_{i=0}^{m-1}\binom{n}{i} \approx \dfrac{1}{(m-1)!}n^{m-1},\mbox{ }n\geq 1.$
\end{enumerate}

\end{lema}

    

Now we consider $\mathcal{G}_2=\langle 1,e_1,e_2 \mid e_ie_j=-e_je_i  \rangle$ the subalgebra of the infinite-dimensional Grassmann algebra generated by $1,e_1$ and $e_2$. On this algebra we define the following involutions:
	\begin{align*}
		\psi( e_i)=e_i,\,\,\,\,\, \tau( e_i)=-e_i,\,\,\,\,\, \gamma( e_i)=(-1)^ie_i,\mbox{ for all }i=1,2.
	\end{align*}

\begin{enumerate}
			\item[$\bullet$] For $g,h\in G$ we denote by $\mathcal{G}_{2,*}^{g,h}$ as the algebra $\mathcal{G}_2$ endowed with the involution $*\in \{\psi, \tau, \gamma\} $ and the only $G$-grading such that
$$1\in (\mathcal{G}_{2,*}^{g,h})_1,\, e_{1}\in (\mathcal{G}_{2,*}^{g,h})_g,\, e_{2}\in (\mathcal{G}_{2,*}^{g,h})_h \mbox{ and }e_{1}e_2\in (\mathcal{G}_{2,*}^{g,h})_{gh}.$$ \end{enumerate}

The description of the ideals of $(G,*)$-identities of the $(G,*)$-algebras defined above can be seen in \cite{WesAnaRaf}, \cite{lamattinamisso} and \cite{MN}. Also, we have information about the $(G,*)$-codimensions given in the following lemmas. 

	\begin{lema} \cite{WesAnaRaf, MN} \label{18} Consider $g\in G$ with $|g|> 2$ and $h\in G$ with $|h|=2$ (in case $|G|$ is even). Then, 
		
		\begin{enumerate}	
			\item[1)] $c_n^{(G,*)}(\mathcal{G}_{2,\circ}^{g,g})=1+2n+\displaystyle\binom{n}{2}$ \;and\; $c_n^{(G,*)}(\mathcal{G}_{2,\gamma}^{g,g})=1+3n+2\displaystyle\binom{n}{2}$, for all $\circ \in \{\tau, \psi\}$.

			\item[2)] $c_n^{(G,*)}(\mathcal{G}_{2,\circ}^{g,g^{-1}})=1+3n+2\displaystyle\binom{n}{2}$ \;and\; $c_n^{(G,*)}(\mathcal{G}_{2,\gamma}^{g,g^{-1}})=1+2n+2\displaystyle\binom{n}{2}$,  for all  $\circ \in \{\tau, \psi\}$.

            \item[3)] { $c_n^{(G,*)}(\mathcal{G}_{2,\gamma}^{h,h})=1+2n+\displaystyle 2\binom{n}{2}$ and $c_n^{(G,*)}(\mathcal{G}_{2,\diamond}^{h,h})=1+2n+\displaystyle\binom{n}{2}, \; for \; all \;\diamond\in \{\psi, \tau\}.$}

		\end{enumerate}
	\end{lema}  
  
	\begin{lema}\cite{WesAnaRaf, lamattinamisso, MN} Consider $g,h,s\in G- \{1\}$ such that $gh\neq 1$. Then,

\begin{enumerate}
    \item[1)] $c_n^{(G,*)}(\mathcal{G}_{2,\tau}^{1,1})=1+n+\displaystyle\binom{n}{2}\;\; and \;\;c_n^{(G,*)}(\mathcal{G}_{2,\tau}^{1,s})=c_n^{(G,*)}(\mathcal{G}_{2,\gamma}^{1,s})=1+2n+\displaystyle\binom{n}{2}.$

 \item[2)] $c_n^{(G,*)}(\mathcal{G}_{2,\circ}^{g,h})=1+3n+2\displaystyle\binom{n}{2}, \;\mbox{for all}\; \circ \in \{\tau, \gamma, \psi\}.$

\end{enumerate}

	\end{lema}




Consider the subalgebra of $UT_4$ given by $$W=F(e_{11}+\cdots + e_{44})+F(e_{12}+e_{34})+F(e_{13}+e_{24})+Fe_{14}.$$
	On this algebra we consider the trivial involution $\nu_1$ and the involutions $\nu_2$ and $\nu_3$ defined below
	
	$$ \begin{pmatrix}
		a& b & c & d \\
		0&a & 0 & c\\
		0&0 & a & b \\
		0&0 & 0 & a 
	\end{pmatrix}^{\nu_2}= \begin{pmatrix}
		a& -b & -c & d \\
		0&a & 0 & -c\\
		0&0 & a & -b \\
		0&0 & 0 & a 
	\end{pmatrix} \mbox{ and } \begin{pmatrix}
		a& b & c & d \\
		0&a & 0 & c\\
		0&0 & a & b \\
		0&0 & 0 & a 
	\end{pmatrix}^{\nu_3}= \begin{pmatrix}
		a& -b & c & -d \\
		0&a & 0 & c\\
		0&0 & a & -b \\
		0&0 & 0 & a 
	\end{pmatrix}.$$

\begin{enumerate}
			\item[$\bullet$] For $g,h\in G$ we denote by $W_{*}^{g,h}$ as the algebra $W$ endowed with the involution $*\in \{\nu_1, \nu_2, \nu_3\} $ and the only $G$-grading such that 
$$e_{11}+\cdots + e_{44}\in (W_{*}^{g,h})_1,\, e_{12}+e_{34}\in (W_{*}^{g,h})_g,\, e_{13}+e_{24}\in (W_{*}^{g,h})_h \mbox{ and }e_{14}\in (W_{*}^{g,h})_{gh}.$$ \end{enumerate}

In \cite{WesAnaRaf}, the authors gave the generators of the ideals  of $(G,*)$-identities of the $(G,*)$-algebras defined above and proved the following lemmas.

	\begin{lema} \cite{WesAnaRaf} Let $g,h\in G- \{1\}$ with $|g|>2$ and $|h|=2$ (in case $|G|$ is even). Then, 
		\begin{enumerate}
 	
			\item[1)] $c_n^{(G,*)}(W_{\nu_3}^{g,g})=c_n^{(G,*)}(W_{\nu_3}^{g,g^{-1}})=1+3n+{2}\displaystyle\binom{n}{2}$ and $c_n^{(G,*)}(W_{\nu_i}^{g,g^{-1}})=1+2n+{2}\displaystyle\binom{n}{2}$, $i=1,2$.
\item[2)] {
$c_n^{\sharp}(A)=1+3n+{2}\displaystyle\binom{n}{2}$, for all $A\in \{W_{\nu_2}^{1,g},W_{\nu_3}^{1,g},W_{\nu_3}^{h,h} \}.$}

\end{enumerate}

\end{lema}


	\begin{lema} \cite{WesAnaRaf}\label{ultimolema} For distinct elements  $g,h\in G- \{1\}$ such that $gh\neq 1$, we have
 $c_n^{(G,*)}(W_{\nu_i}^{g,h})=1+3n+{2}\displaystyle\binom{n}{2}$, $i=1,2,3.$
	
	\end{lema}

Next, we introduce new $(G,*)$-algebras which will be important to get our desired classification. First, we consider the subalgebras of $UT_3$ defined by
$$M_4= F(e_{11}+e_{33})+Fe_{12}+Fe_{23}+Fe_{13}\mbox{ and }M_5= Fe_{22}+Fe_{12}+Fe_{23}+Fe_{13}.$$

For $g\in G$, we consider the following $(G,*)$-algebras:

\begin{enumerate}
			\item[$\bullet$]  $M_{4,\rho}^{g}$ is the algebra $M_4$ endowed with reflection involution and the only $G$-grading such that $$e_{11}+e_{33}\in (M_{4,\rho}^{g})_1,\,\, e_{12}, e_{23}\in (M_{4,\rho}^{g})_g\mbox{ and }e_{13}\in (M_{4,\rho}^{g})_{g^2}.$$

\item[$\bullet$] $M_{5,\rho}^{g}$ is the algebra $M_5$ endowed with reflection involution and the only $G$-grading such that $$e_{22}\in (M_{5,\rho}^{g})_1,\,\, e_{12}, e_{23}\in (M_{5,\rho}^{g})_g\mbox{ and }e_{13}\in (M_{5,\rho}^{g})_{g^2}.$$ \end{enumerate} \vspace{0.1cm}

Now, consider the following subalgebras of $UT_4$ 
$$M_6=F(e_{11}+e_{44})+Fe_{12}+Fe_{13}+Fe_{14}+Fe_{24}+Fe_{34}\mbox{ and }$$ $$M_7=F(e_{22}+e_{33})+Fe_{12}+Fe_{13}+Fe_{14}+Fe_{24}+Fe_{34}.$$

 On $M_6$ and $M_7$ we define the reflection involution $\rho$ and the involutions $\omega_1$ and $\omega_2$, respectively, as below:

$$\begin{pmatrix}
a& b & c & d\\
0&0 & 0 & e\\
0&0 & 0 & f \\
0&0 & 0 & a
\end{pmatrix}^{\omega_1}=\begin{pmatrix}
a& -f & e & -d\\
0&0 & 0 & c\\
0&0 & 0 & -b \\
0&0 & 0 & a
\end{pmatrix}\mbox{ and }\begin{pmatrix}
0& b & c & d\\
0&a & 0 & e\\
0&0 & a & f \\
0&0 & 0 & 0
\end{pmatrix}^{\omega_2}=\begin{pmatrix}
0& -f & e & -d\\
0&a & 0 & c\\
0&0 & a & -b \\
0&0 & 0 & 0
\end{pmatrix}  .$$ \vspace{0.1cm}

For $g,h\in G$ we define the following $(G,*)$-algebras:

\begin{enumerate}
			\item[$\bullet$] $M_{6,*}^{g,h}$ is the algebra $M_6$ endowed with the involution $*\in \{\rho,\omega_1\} $ and the only $G$-grading such that 
$$e_{11}+e_{44}\in (M_{6,*}^{g,h})_1,\, e_{12},e_{34}\in (M_{6,*}^{g,h})_g,\, e_{13},e_{24}\in (M_{6,*}^{g,h})_h \mbox{ and }e_{14}\in (M_{6,*}^{g,h})_{gh}.$$

\item[$\bullet$] $M_{7,*}^{g,h}$ is the algebra $M_7$ endowed with the involution $*\in \{\rho,\omega_2\}$ and the only $G$-grading such that 
$$e_{22}+e_{33}\in (M_{7,*}^{g,h})_1,\, e_{12},e_{34}\in (M_{7,*}^{g,h})_g,\, e_{13},e_{24}\in (M_{7,*}^{g,h})_h \mbox{ and }e_{14}\in (M_{7,*}^{g,h})_{gh}.$$
\end{enumerate} 

Now we introduce a new notation to make the construction of different gradings simpler and give new examples of $(G,*)$-algebras. In this context the algebra $A_3$, defined at the beginning of this section, will be denoted by $M_8$  and we define the algebra $M_9$ as below:
$$M_8= F(e_{11}+e_{66})+F(e_{23}+e_{45})+Fe_{12}+Fe_{13}+Fe_{46}+Fe_{56}\subset UT_6,$$
$$M_9= F(e_{11}+e_{66})+F(e_{23}-e_{45})+Fe_{12}+Fe_{13}+Fe_{46}+Fe_{56}\subset UT_6.$$

For $g,h\in G$, we define the following $(G,*)$-algebras:

\begin{enumerate}
			\item[$\bullet$] $M_{8,\rho}^{g,h}$ is the algebra $M_8$ endowed with reflection involution and the only $G$-grading such that 
$$e_{11}+e_{66}\in (M_{8,\rho}^{g,h})_1,\, e_{12},e_{56}\in (M_{8,\rho}^{g,h})_g,\, e_{23}+e_{45}\in (M_{8,\rho}^{g,h})_h, \,e_{13}\in (M_{8,\rho}^{g,h})_{gh}\mbox{ and } e_{46}\in (M_{8,\rho}^{g,h})_{gh}.$$

\item[$\bullet$] $M_{9,\rho}^{g,h}$ is the algebra $M_9$ endowed with reflection involution and the only $G$-grading such that 
$$e_{11}+e_{66}\in (M_{9,\rho}^{g,h})_1,\, e_{12},e_{56}\in (M_{9,\rho}^{g,h})_g,\, e_{23}-e_{45}\in (M_{9,\rho}^{g,h})_h, \,\mbox{ and } \, e_{13},e_{46}\in (M_{9,\rho}^{g,h})_{gh}.$$ \end{enumerate}

Next we consider the following subalgebras of $UT_6$ 
$$M_{10}= F(e_{11}+e_{22}+e_{55}+e_{66})+F(e_{12}-e_{56})+Fe_{13}+Fe_{23}+Fe_{45}+Fe_{46},$$
$$M_{11}= F(e_{11}+e_{22}+e_{55}+e_{66})+F(e_{12}+e_{56})+Fe_{13}+Fe_{23}+Fe_{45}+Fe_{46}.$$

For $g,h\in G$ we consider the following $(G,*)$-algebras:

\begin{enumerate}
			\item[$\bullet$] $M_{10,\rho}^{g,h}$ is the algebra $M_{10}$ endowed with reflection involution and the only $G$-grading such that 
$$e_{11}+e_{22}+e_{55}+e_{66}\in (M_{10,\rho}^{g,h})_1,\, e_{12}-e_{56}\in (M_{10,\rho}^{g,h})_g,\, e_{23},e_{45}\in (M_{10,\rho}^{g,h})_h \mbox{ and }e_{13},e_{46}\in (M_{10,\rho}^{g,h})_{gh}.$$

\item[$\bullet$] $M_{11,\rho}^{g,h}$ is the algebra $M_{11}$ endowed with reflection involution and the only $G$-grading such that 
$$e_{11}+e_{22}+e_{55}+e_{66}\in (M_{11,\rho}^{g,h})_1,\, e_{12}+e_{56}\in (M_{11,\rho}^{g,h})_g,\, e_{23},e_{45}\in (M_{11,\rho}^{g,h})_h \mbox{ and }e_{13},e_{46}\in (M_{11,\rho}^{g,h})_{gh}.$$ \end{enumerate}

By Lemma \ref{codimensoes}, $c_n^*(A)\leq c_n^{(G,*)}(A)$, for all $n$, and so according to \cite[Lemmas 20-25]{lamattinamisso} we have the following.

\begin{lema} 
    For all $A\in \{M_{i,\rho}^{g}, M_{j, \rho}^{g,h}, M_{6, \omega_1}^{g,h},M_{7, \omega_2}^{g,h}, \mid g,h\in G, i=4,5; j=6,7,9,10\}$ and $n\geq 3$ we have $$n(n-1)\leq c_n^{*}(A)\leq c_n^{(G,*)}(A)\,\,\,\mbox{ and }\,\,\,\frac{(n-1)(n-2)}{2}\leq c_n^{*}(M_{8, \rho}^{g,h})\leq c_n^{(G,*)}(M_{8, \rho}^{g,h}).$$
\end{lema}

For the algebra $M_{11,\rho}^{g,h}$ we have the following.

\begin{lema}
For all $g,h\in G$ with $g\neq 1$ and $n\geq 3 $ we have $n(n-1)\leq c_n^{(G,*)}(M_{11,\rho}^{g,h}).$    
\end{lema}

\begin{proof}

In fact, for $g,h\in G$ with $g\neq 1$ we consider the multipartition $\gen{\lambda}=((n-2)_{1^{+}}, (1)_{g^+}, (1)_{h^{-}})$ and the respective multitableau $T_{\gen{\lambda}}$ below
  $$\left(\begin{array}{l}
			\begin{array}{|c| c| c | }
				\hline
				\hspace{0.1cm} 1 & 	\hspace{0.1cm} \cdots  &  n-2	\hspace{0.1cm} \\
				\hline
			\end{array}\,\, {_{1^+}}, \,\,\,   \begin{array}{|c| }
			\hline
			\hspace{0.1cm} n-1  \\
			\hline
		\end{array}\,\,_{g^{+}}, \,\,\,\begin{array}{|c| }
			\hline
			\hspace{0.1cm} n  \\
			\hline
		\end{array}\,\,_{h^{-}}\end{array} \right)$$ with highest weight vectors is $f_{T_{\gen{\lambda}}}=y_{1,1}^{n-1}y_{2,g}z_{3,h}$. Considering the evaluation $y_{1,1}\mapsto e_{11}+e_{22}+e_{55}+e_{66}$, $y_{2,g}\mapsto e_{12}+e_{56}$ and $z_{3,h}\mapsto e_{23}-e_{45}$ we get $f_{T_{\gen{\lambda}}}\notin \textnormal{Id}^{(G,*)}(A)$ and so $m_{\gen{\lambda}}\neq 0$. Finally, by (\ref{293}), we have $$c_n^{(G,*)}(M_{11,\rho}^{g,h})\geq \binom{n}{n-2,1,1}d_{(n-2)}d_{(1)}d_{(1)}= n(n-1).$$
\end{proof}

Since the nilpotency index of the Jacobson radical of the algebras $M_{i}$, $i=6, \ldots , 11$, is equal to $3$ then, as a consequence of the previous lemmas and Proposition \ref{prop1}, we have the following.

\begin{corolario}\label{corolarioquad}
    For all $A\in \{M_{i,\rho}^{g}, M_{j, \rho}^{g,h}, M_{6, \omega_1}^{g,h},M_{7, \omega_2}^{g,h}, M_{11,\rho}^{v,h} \mid g,h,v\in G, v\neq 1; i=4,5; j=6,\ldots ,10\}$ we have $c_n^{(G,*)}(A)\approx \alpha n^2$, for some $\alpha >0$ depending on $A$.
\end{corolario}

 \section{Decomposing the Jacobson radical}

In this section we focus our attention in the finite-dimensional $(G,*)$-algebras of type $A=F+J(A)$. Here, we use the previous $(G,*)$-algebras in order to get an useful decomposition of $A$ into a direct sum of well understood  algebras.

We start recalling (see \cite[Lemma 2]{Giambruno}) that if $A=F+ J(A)$ is a finite-dimensional \Gstar algebra over $F$ then $J(A)$ is a $T_{(G,*)}$-ideal of $A$ (see \cite[Theorem 2.8]{Lorena}) and can be decomposed into the direct sum of $G$-graded subspaces
	\begin{equation} \label{6}
		J(A)=J_{00}+J_{10}+J_{01}+J_{11},\mbox{ where }
	\end{equation} $$J_{00}=\{j\in J\mid 1_Fj=j1_F=0\},\quad J_{10}=\{j\in J\mid 1_Fj=j 
    \mbox{ and }j1_F=0\}$$
      $$J_{01}=\{j\in J\mid 1_Fj=0 
    \mbox{ and }j1_F=j\}\quad \mbox{ and }\quad J_{11}=\{j\in J\mid 1_Fj=j1_F=j\}.$$ Moreover, the subspaces $J_{00}$ and $J_{11}$ are invariant under the involution, $(J_{10})^{*}=J_{01}$ and for $i, k,r,s \in \{0, 1\}$ we have $J_{ik}J_{rs} \subset \delta_{kr} J_{is}$, where $\delta_{kr}$ is the Kronecker delta function. Therefore, we may decompose $J_{00}$ and $J_{11}$ into symmetric and skew components as follow: 
	$$J_{00}= \underset{g\in G}{\sum} ((J_{00})_g^{+}+ (J_{00})_g^{-})\mbox{ and }J_{11}= \underset{g\in G}{\sum} ((J_{11})_g^{+}+ (J_{11})_g^{-}).$$

Consider the following sets of $(G,*)$-algebras
$$\mathcal{I}_1= \{M_{4,\rho}^g\mid g\in G\}\,\,\,\mbox{ and }\,\,\,\mathcal{I}_2= \{M_{5,\rho}^g\mid g\in G\}.$$

\begin{lema} \label{lema1} For the $(G,*)$-algebra $A = F + J(A)$ we have,
    \begin{enumerate}
        \item[1)] If $Q \not\in \textnormal{var}^{(G,*)}(A)$, for all $Q\in \mathcal{I}_1$, then $aa^*=0$, for all $a\in (J_{10})_g$ and for all $g\in G.$ 
        \item[2)] If $Q \not\in \textnormal{var}^{(G,*)}(A)$, for all $Q\in \mathcal{I}_2$, then $aa^*=0$, for all $a\in (J_{01})_g$ and for all $g\in G.$ 
    \end{enumerate}
\end{lema}

\begin{proof}
  Let $g\in G$ and suppose $aa^*\neq 0$, for some $a\in (J_{10})_g$. Consider $R$ the $(G,*)$-subalgebra of $A$ generated by $1_F,a$ and $a^*$ and let $I$ be the $T_{(G,*)}$-ideal of $R$ generated by $a^*a$. Since $a,a^*,aa^*\notin I$ then $R/I$ is linearly generated by $\overline{1_F}$, $\overline{a}$, $\overline{a^*}$ and $\overline{aa^*}$ and so the map $\varphi:R/I\rightarrow M_{4,\rho}^g$ given by 
    $\varphi(\overline{1_F})=e_{11}+e_{33}$, $\varphi(\overline{a})=e_{12}$, $\varphi(\overline{a^*})=e_{23}$ and $\varphi(\overline{aa^*})=e_{13}$ defines an isomorphism of $(G,*)$-algebras.
    
    The second item is proved similarly.

\end{proof}

Define the following sets $$\mathcal{I}_3= \mathcal{I}_1\cup \mathcal{I}_2\cup \{M_{8,\rho}^{g,h}\mid g,h\in G\}\,\,\mbox{ and }\,\,\mathcal{I}_4= \mathcal{I}_1\cup \mathcal{I}_2\cup \{M_{9,\rho}^{g,h}\mid g,h\in G\}.$$

\begin{lema} \label{lema2} For the $(G,*)$-algebra $A = F + J(A)$ we have,
\begin{enumerate}
    \item[1)] If $Q \not\in \textnormal{var}^{(G,*)}(A)$, for all $Q\in \mathcal{I}_3$, then {$J_{10}(J_{00})_g^{+}=(J_{00})_g^{+}J_{01}=0$}, for all $g\in G$. 
    \item[2)] If $Q \not\in \textnormal{var}^{(G,*)}(A)$, for all $Q\in \mathcal{I}_4$, then {$J_{10}(J_{00})_g^{-}=(J_{00})_g^{-}J_{01}=0$}, for all $g\in G$.
\end{enumerate}
\end{lema}

\begin{proof}

Assume that $J_{10}(J_{00})_g^{+}\neq 0$ (resp. $J_{10}(J_{00})_g^{-}\neq 0$) and consider $a\in (J_{10})_h$ and $b\in (J_{00})_g^{+}$ (resp. $ b\in (J_{00})_g^{-}$) such that $ab\neq 0$, for some $g,h\in G.$ Let $R$ be the $(G,*)$-subalgebra of $A$ generated by $1_F, a,a^*$ and $b$ with induced $G$-grading and induced involution. Since $Q\notin \textnormal{var}^G(A)$, for all $Q\in \mathcal{I}_1$, then we have $aa^*=a^*a=0$. Moreover, by construction, we also have $a^2=(a^*)^2=0$. Consider $I$ the $T_{(G,*)}$-ideal of $R$ generated by $b^2$ and $aba^*$ and notice that it is linearly generated by $aba^*, ab^2, b^2a^*,ab^2a^*, b^2, b^3 , \ldots, b^k$, for some $k\geq 1$. An easy calculation proves that $a,a^*,b,ab$ and $ba^*$ do not belong to $I$ and so $R/I$ is a $(G,*)$-algebra linearly generated by $\overline{a},\overline{a^*},\overline{b},\overline{ab}$ and $\overline{ba^*}$. Finally,

\begin{enumerate}
    \item[1.] if $b\in (J_{00})_g^{+}$ then the map $\varphi_1: R/I\rightarrow M_{8,\rho}^{h,g}$ defined by $\varphi_1(\overline{1}_F)=e_{11}+e_{66},\, \varphi_1(\overline{a})=e_{12}, \,\varphi_1(\overline{a^*})=e_{56}, \,\varphi_1(\overline{b})=e_{23}+e_{45}, \,\varphi_1(\overline{ab})=e_{13}\mbox{ and } \varphi_1(\overline{ba^*})=e_{46}$ is an isomorphism of $(G,*)$-algebras and so $M_{8,\rho}^{h,g}\in \textnormal{var}^G(A)$,

    \item[2.] if $b\in (J_{00})_g^{-}$ then the map $\varphi_2: R/I\rightarrow M_{9,\rho}^{h,g}$ defined by $\varphi_2(\overline{1}_F)=e_{11}+e_{66},\, \varphi_2(\overline{a})=e_{12}, \,\varphi_2(\overline{a^*})=e_{56}, \,\varphi_2(\overline{b})=e_{23}-e_{45}, \,\varphi_2(\overline{ab})=e_{13}\mbox{ and } \varphi_2(\overline{ba^*})=-e_{46}$ is an isomorphism of $(G,*)$-algebras and so $M_{9,\rho}^{h,g}\in \textnormal{var}^G(A).$
\end{enumerate}   

These contradictions prove that $J_{10}(J_{00})_g^{+}=J_{10}(J_{00})_g^{-}=0$. Since $(J_{00})_g^{+}J_{01}=(J_{10}(J_{00})_g^{+})^*=0$ and $(J_{00})_g^{-}J_{01}=-(J_{10}(J_{00})_g^{-})^*=0$ then the proof follows.

\end{proof}

\begin{observacao} \label{obs1}
    Notice that if $a\ne 0$ is an element of a $(G,*)$-algebra $A$ such that $a^*= \alpha a$, for some $\alpha\neq 0$, then we must have $a=(a^*)^*=(\alpha a)^*= \alpha a^*=\alpha^2 a$ and so $a^*=a$ or $a^*=-a$. 
\end{observacao}

Consider the following sets of $(G,*)$-algebras $$\mathcal{I}_5= \mathcal{I}_3\cup \mathcal{I}_4\cup\{M_{6,\omega_1}^{g,h}, M_{6,\rho}^{r,s}\mid g,h,r,s\in G, r\neq s\}\,\mbox{ and }\,$$
$$\mathcal{I}_6= \mathcal{I}_3\cup \mathcal{I}_4\cup\{M_{7,\omega_2}^{g,h}, M_{7,\rho}^{r,s}\mid g,h,r,s\in G,r\neq s\}.$$

\begin{lema} \label{lema3} For the $(G,*)$-algebra $A = F + J(A)$ we have,

\begin{enumerate}
    \item[1)] If $Q \not\in \textnormal{var}^{(G,*)}(A)$, for all $Q\in \mathcal{I}_5$, then $J_{10}J_{01}=0$.
     \item[2)] If $Q \not\in \textnormal{var}^{(G,*)}(A)$, for all $Q\in \mathcal{I}_6$, then $J_{01}J_{10}=0$.
\end{enumerate}
\end{lema}

\begin{proof}
 Suppose that there exist $a\in (J_{10})_g$ and $b\in (J_{01})_h$ such that $ab\neq 0$, for some $g,h\in G$. Let $R$ be the $(G,*)$-subalgebra of $A$ generated by $1_F,a,a^*,b$ and $b^*$. By construction and using Lemmas \ref{lema1} and \ref{lema2}, we have $c=c^*=0$, for all $c\in \{a^2,b^2,aba,bab,aa^*,a^*a,b^*b,bb^*,ab^*,a^*b\}$ and so $R$ is linearly generated by $1_F,a,a^*,b,b^*,ab,ba,(ab)^*$ and $(ba)^*$. Let $I$ be the $T_{(G,*)}$-ideal generated by $ba$ and notice that it is linearly generated by $ba$ and $(ba)^*$ and so $a,a^*,b,b^*,ab,(ab)^*\notin I$. Thus, $R/I$ is linearly generated by $\overline{1_F},\overline{a},\overline{a^*},\overline{b},\overline{b^*},\overline{ab}$ and $\overline{(ab)^*}$.

{\bf Case 1}: $g=h$.

In this case, we notice that $\overline{a}+\overline{b^*}\in (J_{10})_g$ and so, by Lemma \ref{lema1}, we have $(\overline{a}+\overline{b^*})(\overline{a}+\overline{b^*})^*=0$. Since $\overline{aa^*}=\overline{b^*b}=0$ we have $\overline{(ab)^*}=\overline{b^*a^*}=-\overline{ab}$ then $\overline{ab}\in (\overline{J_{11}})_{g^2}^-$. Notice that $R/I\cong M_{6,\omega_1}^{g,g}$ through the linear map $\varphi_1(\overline{1_F})=e_{11}+e_{44}$, 
$\varphi_1(\overline{a})=e_{12}$, $\varphi_1(\overline{a^*})=-e_{34}$, $\varphi_1(\overline{b})=e_{24}$, 
$\varphi_1(\overline{b^*})=e_{13}$ and $\varphi_1(\overline{ab})=e_{14}$.

{\bf Case 2}: $g\neq h$.

In this case, if $\overline{(ab)^*}=\alpha \overline{ab}$, for some $\alpha\in F$, then, by Remark \ref{obs1}, we must have $\overline{(ab)^*}=-\overline{ab}$ or $\overline{(ab)^{*}}=\overline{ab}$. In the first case, the map $\varphi_1$ defined above defines an isomorphism of $(G,*)$-algebras between $R/I$ and $M_{6,\omega_1}^{g,h}$. In the second case, we notice that the linear map $\varphi_2: R/I\rightarrow M_{6,\rho}^{g,h}$ defined by $\varphi_2(\overline{1_F})=e_{11}+e_{44}$, 
$\varphi_2(\overline{a})=e_{12}$, $\varphi_2(\overline{a^*})=e_{34}$, $\varphi_2(\overline{b})=e_{24}$, 
$\varphi_2(\overline{b^*})=e_{13}$, $\varphi_2(\overline{ab})=e_{14}$ is an isomorphism of $(G,*)$-algebras. 

Assume that $\overline{(ab)^*}\notin F \overline{ab}$. Define the subspace $K=F( \overline{ab} +\overline{(ab)^*})$ and notice that it is a $T_{(G,*)}$-ideal of $R/I$. If $\overline{ab}\in K$ then we get $\overline{(ab)^*}=\beta \overline{ab}$, for some $\beta \in F$, a contradiction. Therefore, we can check that $(R/I)/K$ is linearly generated by the non-zero elements $\overline{\overline{1_F}},\overline{\overline{a}},\overline{\overline{a^*}},\overline{\overline{b}},\overline{\overline{b^*}}$ and $\overline{\overline{ab}}$ where $\overline{\overline{ab}}\in (\overline{\overline{J_{11}}})_{gh}^-$. Similarly, we conclude that $(R/I)/K\cong M_{6,\omega_1}^{g,h}$.

A similar approach can be applied in the proof of item $2).$
\end{proof}

The next lemma was proved by Cota, dos Santos and Vieira in \cite{WesAnaRaf}.

\begin{lema} \label{lema5} Let $A=F+J_{11}$ be a $(G,*)$-algebra and $g\in G$.
\begin{enumerate}
    \item[1)] If $C_{k}^g\notin \textnormal{var}^{(G,*)}(A)$ then $b^{k-1}=0$, for all $b\in (J_{11})_g^+ $.

    \item[2)] If $ C_{k,*}^g\notin \textnormal{var}^{(G,*)}(A)$ then $b^{k-1}=0$, for all $b\in (J_{11})_g^-$.
\end{enumerate}
	\end{lema}

Let $\mathcal{I}_7$ and $\mathcal{I}_8$ be the following sets $$\mathcal{I}_7= \mathcal{I}_6\cup \{C_{3,*}^{g},M_{10,\rho}^{g,h}\mid g,h\in G\}\,\mbox{ and }\,$$
$$\mathcal{I}_8= \mathcal{I}_6\cup\{C_{3}^{g},M_{11,\rho}^{g,h}\mid g,h\in G, g\neq 1\}.$$

\begin{lema} \label{lema4} For the $(G,*)$-algebra $A = F + J(A)$ we have,

\begin{enumerate}
    \item[1)] If $Q \not\in \textnormal{var}^{(G,*)}(A)$, for all $Q\in \mathcal{I}_7$, then $J_{01}(J_{11})_g^-=(J_{11})_g^-J_{10}=\{0\}$, for all $g\in G$.
     \item[2)] If $Q \not\in \textnormal{var}^{(G,*)}(A)$, for all $Q\in \mathcal{I}_8$, then $J_{01}(J_{11})_g^+=(J_{11})_g^+J_{10}=\{0\}$, for all $g\in G- \{1\}$.
\end{enumerate}
\end{lema}

\begin{proof}
    Suppose that $J_{01}(J_{11})_g^-\ne \{0\}$ and consider $a\in (J_{01})_h$ and $b\in (J_{11})_g^-$ such that $ab\neq 0$, for some $g,h\in G$. Let $R$ be the $(G,*)$-subalgebra of $A$ generated by $1_F,a,a^*$ and $b$. By the choice of $a$ and Lemmas \ref{lema5} and \ref{lema1} we have $aa^*=a^*a=a^2=b^2=ba=a^*b=0$. Moreover, since $Q\notin \textnormal{var}^{(G,*)}(A)$, for all $Q\in \mathcal{I}_6$, by Lemma \ref{lema3} we have $aba^*=0$ and so $R$ is linearly generated by $1_F,a,a^*,b,ab$ and $ba^*$. Therefore, $R\cong M_{10,\rho}^{g,h}$ through the map $\varphi({1_F})=e_{11}+e_{22}+e_{55}+e_{66}$, 
$\varphi({a})=e_{45}$, $\varphi({a^*})=e_{23}$, $\varphi({b})=e_{12}-e_{56}$, $\varphi({ab})=-e_{46}$ and $\varphi(ba^*)=e_{13}.$

The second item is proved similarly.
\end{proof}

  Consider the following set of $(G, \ast)$-algebras $$\mathcal{I}_9=\{C_{3}^{g}, C_{3,*}^{u},U_{3}^*, \mathcal{G}_{2,\diamond}^{g,h}, W_{\nu}^{g,h},\mathcal{G}_{2,\imath}^{1,h}, W_{\eta}^{1,h} \}$$ where $g,h,u\in G$ with $g,h\neq 1$, $\diamond,\imath\in \{\psi, \tau, \gamma\}$ with $ \imath \neq \psi$ and $\nu,\eta\in \{\nu_1,\nu_2,\nu_3\}$ with $ \eta\neq \nu_1$.

	\begin{lema} \label{10}
		Let $A=F+J_{11}$ be a $(G,*)$-algebra such that $Q\notin \textnormal{var}^{(G,*)}(A)$, for all $Q\in \mathcal{I}$. Then, $[(J_{11})_1^+,(J_{11})_1^+]=(J_{11})_1^{-}(J_{11})_g=(J_{11})_g(J_{11})_1^{-}=(J_{11})_g(J_{11})_h=\{0\}$, for all $g,h\in G-\{1\}$.	\end{lema}

\begin{proof}

Since $F+(J_{11})_1$ is a $(G,*)$-subalgebra of $A$ with trivial $G$-grading and $U_{3}^*\notin \textnormal{var}^{(G,*)}(F+(J_{11})_1)$ then by \cite[Lemma 4.7]{MN}  we have $[(J_{11})_1^+,(J_{11})_1^+]=\{0\}.$

	Suppose that there exist symmetric or skew elements $a\in (J_{11})_g$ and $b\in (J_{11})_h$ such that $ab\neq 0$, for some $g,h\in G$ with $a,b\notin (J_{11})_1^+$. Let $R$ be the $(G,*)$-subalgebra of $A$ generated by $1_F,a$ and $b$ with induced $G$-grading and induced involution. Since $C_{3}^{r}, C_{3,*}^{s}\notin \textnormal{var}^{(G,*)}(A),$ for all $r,s\in G$ with $r\neq 1$, then by Lemma \ref{lema5} we have $a^2=b^2=0$ for any choice of $a$ and $b$.

 {\bf Case 1:} $g=h$ with $g\neq 1$.

 If both $a$ and $b$ are symmetric or both are skew we also have $(a+b)^2=0$ and so $ab=-ba$. Therefore, the map $\varphi_1: R \rightarrow \mathcal{G}_2$ given by $\varphi_1({1_F})= 1, \varphi_1({a})= e_1, \varphi_1({b})=e_2$ and $\varphi_1({ab})=e_1e_2$ defines an isomorphism of $(G,*)$-algebras in the following cases: 

\begin{enumerate}
			\item[1.] if $a\in (J_{11})_g^+$, $b\in (J_{11})_g^+$ then $R\cong \mathcal{G}_{2,\psi}^{g,g};$ 
			\item[2.] if $a\in (J_{11})_g^-$, $b\in (J_{11})_g^-$ then $R\cong \mathcal{G}_{2,\tau}^{g,g}.$
		\end{enumerate}

  Therefore, we assume that $a$ is skew and $b$ is symmetric. If $ba\in Fab$, by Remark \ref{obs1} either $ba=-ab$ or $ba=ab$. In the first case, the map $\varphi_1: R\rightarrow \mathcal{G}_2$ defined above is an isomorphism of $(G,*)$-algebras between $R$ and $\mathcal{G}_{2,\gamma}^{g,g}.$ In the second case, the map $\varphi_2: R \rightarrow W$ given by $\varphi_2({1_F})= e_{11}+\cdots + e_{44}, \varphi_2({b})=e_{12}+e_{34}, \varphi_2({a})=e_{13}+e_{24}$ and $\varphi_2({ab})=e_{14}$ defines an isomorphism of $(G,*)$-algebras between $R$ and $W_{\nu_3}^{g,g}$.

  Assume that $ba\notin Fab$ and consider $I=\langle ab+ba \rangle_{T_{(G,*)}}$ the $T_{(G,*)}$-ideal of $R$ generated by $ab+ba$. Since $C_{3,*}^{g}\notin \textnormal{var}^{(G,*)}(A)$ then $I$ is linearly generated by $ab+ba, aba, bab$ and $abab$, and so $a,b$ and $ab$ do not belong to $I$. Therefore, $R/I$ is a $(G,*)$-algebra linearly generated by $\overline{1_F},\overline{a}, \overline{b}, \overline{ab}$ satisfying $\overline{a}^2=\overline{b}^2=\overline{ab}+\overline{ba}=0$. Hence $R/I\cong \mathcal{G}_{2, \gamma}^{g,g}.$

 {\bf Case 2:} $g\neq h$ with $g,h\neq 1$.

If $ba\in Fab$, then by Remark \ref{obs1} either $ab=ba$ or $ab=-ba$. In the first case, according to the choice of $a$ and $b$, the map $\varphi_2: R \rightarrow W$ defined above defines an isomorphism of $(G,*)$-algebras between $R$ and $W_{*}^{g,h}$, for some $*\in \{\nu_1,\nu_2,\nu_3\}$. In the second case, according to the choice of $a$ and $b$, the map $\varphi_1: R \rightarrow \mathcal{G}_2$ defined above defines an isomorphism of $(G,*)$-algebras between $R$ and $\mathcal{G}_{2,*}^{g,h}$, for some $*\in \{\psi, \tau, \gamma\}$.

  Therefore, we may assume that $ba\notin Fab$. In this case, we consider $I$ the $T_{(G,*)}$-ideal of $R$ generated by $ab+ba$.
 A similar approach proves that $a,b,ab\notin I$ and so $R/I$ is linearly generated by $\overline{1_F}$, $\overline{a}$, $\overline{b}$ and $\overline{ab}$ where $\overline{a}^2=\overline{b}^2= \overline{ab}+\overline{ba}=0$. Now, according to the choice of $a$ and $b$, it is easy to check that $R/I \cong \mathcal{G}_{2,*}^{g,h}$, for some $*\in \{\psi, \tau, \gamma\}.$ 
 
  {\bf Case 3:} $a\in (J_{11})_1^{-}$ and $h\neq 1.$

Following the steps of the proof of Case $2$ we may check that if $C_{3,*}^{1}, \mathcal{G}_{2,\tau}^{1,h}, \mathcal{G}_{2,\gamma}^{1,h}, W_{\nu_2}^{1,h}, W_{\nu_3}^{1,h}\notin \textnormal{var}^{(G,*)}(A)$ then we have $(J_{11})_1^{-}(J_{11})_h=(J_{11})_h(J_{11})_1^{-}=\{0\}.$ 

 \end{proof}

Let $\mathcal{I}$ be the set of all the previous algebras, that is, $$\mathcal{I}= \mathcal{I}_5\cup \mathcal{I}_7\cup \mathcal{I}_8\cup \mathcal{I}_9.$$

\begin{corolario} \label{cor1}
   Let $A=F+J(A)$ be a $(G,*)$-algebra such that $Q\notin \textnormal{var}^{(G,*)}(A)$, for all $Q\in \mathcal{I}$, then $[y_{1,1},y_{2,1}]$, $x_{1,g}x_{2,h},z_{1,1}x_{2,g}\in \textnormal{Id}^{(G,*)}(A)$, for all $g,h\in G- \{1\}$.
\end{corolario}

\begin{lema} \label{306}
		Let $A$ be a $(G,*)$-algebra under the hypothesis of Corollary \ref{cor1}. Then, $$A\sim_{T_{(G,*)}} \displaystyle {\bigoplus_{g\in G} }B^g,$$ where $B^g=F +(J_{11})_1^+ +J(A)_g$. Moreover, $B^h\in \textnormal{var}^{(G,*)}(M_{h,\rho})$, for all $h\in G-\{1\}$.
	\end{lema}
	
	\begin{proof}
	First, by Corollary \ref{cor1}, we notice that $B^g$ is a $(G,*)$-subalgebra of $A$, for all $g\in G$. Hence, $\textnormal{Id}^{(G,*)}(A)\subset \textnormal{Id}^{(G,*)}(B)$, where $B=\displaystyle {\bigoplus_{g\in G} }B^g $. Let $f$ be a multilinear $(G,*)$-identity of $B$. After reducing $f$ modulo the $T_{(G,*)}$-ideal
 $$\langle [y_{1,1},y_{2,1}], x_{1,g}x_{2,h},z_{1,1}x_{2,g}\mid g,h\in G- \{1\}\rangle_{T_{(G,*)}} $$ and using the multihomogeneity of $T_{(G,*)}$-ideals, we may assume that $f$ is a linear combination of polynomials of one of the following types
		\begin{enumerate}
		\item[1.] $x_{\sigma(1),1}\cdots x_{\sigma(n),1}$;
	
		\item[2.] $y_{i_1,1}\cdots y_{i_{j-1},1}x_{j,g} y_{i_{j+1},1}\cdots y_{i_n,1}$, $1\leq j\leq n$;
  
	\end{enumerate} 
  for some $\sigma \in S_n$, $x_i\in\{y_i,z_i\}$, $g\in G-\{1\}$ and $i_1< \cdots < i_{j-1}$, $i_{j+1}< \cdots < i_{n}$.
		
		If $f$ is a linear combination of polynomials of the first type then we must evaluate $f$ on elements from the $(G,*)$-subalgebra $B^1$ of $A$ and so $f\equiv 0$ on $A$.	In the second case, we must evaluate $f$ on the $(G,*)$-subalgebra $B^g$ of $A$. Since $f$ is a $(G,*)$-identity of $B^h$, for all $h\in G$, we also have $f\equiv 0$ on $A$ and so $\IdGstar{A}= \IdGstar{B}$.

  Finally, by Example \ref{43}, it is immediate that $B^h\in \textnormal{var}^{(G,*)}(M_{h,\rho})$, for all $h\in G-\{1\}$.
	\end{proof}

 The previous lemma will be crucial in order to get our classification, in fact the previous decomposition seems to be much more comprehensible. Moreover, we notice that $\mathcal{I}$ consists entirely of $(G,*)$-algebras having quadratic growth.

\section{\texorpdfstring{Classifying the minimal $(G,*)$-varieties of quadratic growth}{Classifying the minimal (G,*)-varieties of quadratic growth}}

The main purpose of this section is to classify the minimal $(G,*)$-varieties having quadratic growth of corresponding $(G,*)$-codimension sequence. We present a finite list of $(G,*)$-algebras generating all such $(G,*)$-varieties. 

Recall that a $(G,*)$-variety $\mathcal{V}=\textnormal{var}^{(G,*)}(A)$ is minimal of polynomial growth $n^k$ if $c_n^{(G,*)}(A)\approx \alpha n^k$ and any proper subvariety $\mathcal{U}$ of $\mathcal{V}$ has polynomial growth $ n^p$,  for some $p<k$ and $\alpha > 0$. By convenience, we also say that the $(G,*)$-algebra $A$ is minimal. In \cite{Mara}, de Oliveira, dos Santos and Vieira classified the minimal $(G,*)$-varieties of linear growth and proved that the direct sum of distinct $(G,*)$-algebras generating such varieties generates all $(G,*)$-varieties of linear growth.

\begin{teorema} \label{minimaislinear}  A $(G,*)$-algebra $A$ generates a minimal $(G,*)$-variety of linear growth if and only if $A$ is ${T_{(G,*)}}$-equivalent to one of the following $(G,*)$-algebras: $C_{2,*}^g, A_2^*, C_2^h, A_2^\mathsf{h},$ for some $g,h\in G$ with $h\neq 1.$
\end{teorema}

Before we present our main result, we consider the following remark.

\begin{observacao} \label{12}
		For each $g\in G- \{1\}$ we have
		\begin{enumerate}
						\item[1)] $C_{m,*}^1\in \textnormal{var}^{(G,*)}((FC_2)_{*}),$ for all $m\geq 2.$
			\item[2)] If $|G|$ is even and $h\in G$ with $|h|=2$ then  $C_{m,*}^h\in \textnormal{var}^{(G,*)}((FC_2)^{(G,*)}),$ for all $m\geq 2$.
			\item[3)] For all prime $p\mid  |G|$, $h\in G$ with $|h|=p$ and $C_p=\langle h\rangle$ we have $C_{m}^{h}\in \textnormal{var}^{(G,*)}(FC_p^G),$ for all $m\geq 2$.

   	\item[4)] $N_{3}^*$, $U_{3}^* \in \textnormal{var}^{(G,*)}(M_{1,\rho})$ and $N_{3}^{\mathsf{g}}$, $U_{3}^{\mathsf{g}}$, $A_{2}^{\mathsf{g}} \in \textnormal{var}^{(G,*)}(M_{g,\rho}).$
		\end{enumerate}
	\end{observacao}

We are interested on defining a minimal list of $(G,*)$-algebras of quadratic growth, in a certain sense. The next remark can be consulted in \cite[Lemma 4.10]{mallu}.

\begin{observacao}For the algebras $M_{6,\omega_1}^{1,1}$ and $U_3^*$
we have $\textnormal{Id}^{(G,*)}(M_{6,\omega_1}^{1,1})\subset \textnormal{Id}^{(G,*)}(U_3^*).$ Therefore, if $U_3^*\notin \textnormal{var}^{(G,*)}(A)$ then $M_{6,\omega_1}^{1,1} \notin \textnormal{var}^{(G,*)}(A).$ 
\end{observacao}

Define the following sets $$\mathcal{K}=\{\mathcal{G}_{2,\tau}^{1,1},C_{3,*}^1,N_{3}^*,U_3^*, M_{4, \rho}^{1},M_{5, \rho}^{1},M_{6,\omega_1}^{1,1},M_{7,\omega_2}^{1,1}, M_{j, \rho}^{1,1}\mid 8\leq j\leq 10\},$$
$$\mathcal{K}_1=\mathcal{K}- \{M_{6,\omega_1}^{1,1}\} \,\mbox{ and }\, \mathcal{I}_{10}=\{A_3^{\mathsf{g}},N_3^{\mathsf{g}}, U_3^{\mathsf{g}}\mid g\in G- \{1\}\}$$
and recall that $\mathcal{I}= \mathcal{I}_5\cup \mathcal{I}_7\cup \mathcal{I}_8\cup \mathcal{I}_9.$
Finally, we define the set of $(G,*)$-algebras
$$\mathcal{M}=\mathcal{I}\cup \mathcal{K}_1\cup \mathcal{I}_{10}.$$

 The final goal of this paper is to prove that the $(G,\ast)$-algebras in $\mathcal{M}$ generate the only minimal varieties of quadratic growth of the $(G,\ast)$-codimensions. In fact, by Lemmas \ref{teoAk*} -- \ref{ultimolema} and Corollary \ref{corolarioquad}, for all $A\in \mathcal{M}$ we have $c_n^{(G,*)}(A)\approx qn^2$, for some constant $q>0.$ Also, we observe the following.

\begin{observacao}\label{min} 
Some long and arduous calculations prove that for distinct $(G,\ast)$-algebras $A$ and $B$ in $\mathcal{M},$ we have $\textnormal{Id}^{(G,*)}(A)\not\subset \textnormal{Id}^{(G,*)}(B)$. 
\end{observacao}

Finally, we are ready to prove the main result of this section.
 
\begin{teorema} \label{1234}
 Let $A$ be a finite-dimensional $(G,*)$-algebra over a field $F$ of characteristic zero. The following conditions are equivalent:
 \begin{enumerate}
    
     \item[1)] $Q\notin \textnormal{var}^{(G,*)}(A)$, for all $Q\in  \mathcal{M}$.
     \vspace{0.1 cm}
     
     \item[2)] $A$ is $T_{(G,*)}$-equivalent to $M^{g_1} \oplus \cdots \oplus  M^{g_t}\oplus N$, where for $g_1=1$ we have  $M^{g_1}$ is $T_{(G,*)}$-equivalent to one of the following $(G,*)$-algebras:

     \begin{center}
$N$,  $C\oplus N$,  $C_{2,*}^{1}\oplus N$,  $A_{2}^{\ast}\oplus N$,  $C_{2,*}^{1}\oplus A_{2}^{\ast}\oplus N,$
\end{center} and for all $2\leq i\leq t$ we have $M^{g_i}$ is $T_{(G,*)}$-equivalent to one of the following $(G,*)$-algebras:
 \begin{center} $N$, $C\oplus N$, $C_{2,*}^{{g}_i}\oplus N$, $C_{2}^{{g}_i}\oplus N$, $A_{2}^{\mathsf{g}_i}\oplus N$, \\ $C_{2,*}^{{g}_i} \oplus C_{2}^{{g}_i}\oplus N$,    $C_{2,*}^{{g}_i} \oplus A_{2}^{\mathsf{g}_i}\oplus N$,  $C_{2}^{{g}_i} \oplus A_{2}^{\mathsf{g}_i}\oplus N$,  $C_{2,*}^{{g}_i} \oplus C_{2}^{{g}_i} \oplus A_{2}^{\mathsf{g}_i}\oplus N$, \end{center} where $N$ denotes a nilpotent $(G,*)$-algebra and $C$ is a commutative \Gstar algebra with trivial $G$-grading and trivial involution.

     \item[3)]There exists $\alpha \geq 0$ such that $c_n^{(G,*)}(A)\leq \alpha n$, for all $n\geq 1.$ 
     
 \end{enumerate}
\end{teorema}

\begin{proof}
Since the other implications are immediate, we focus our attention on prove that condition $1)$ implies condition $2)$ and so we assume that $Q\notin \textnormal{var}^{G}(A)$, for all $Q\in  \mathcal{M}$. By Remark \ref{12} and Theorem \ref{polynomialgrowth} we may assume that $A= B_1 \oplus \cdots \oplus B_m$, where for $1\leq i\leq m$, $B_i$ is a finite-dimensional $(G, *)$-algebra over $F$ with either $B_i$ is nilpotent or $B_i\cong F+J(B_i)$. If $B_i$ is nilpotent for all $i$ then we are done. Therefore, we assume that there exists $1\leq i\leq m$ such that $B_i=F+J(B_i),$ where $J(B_i)= J_{00}+J_{10}+J_{01}+J_{11}$ as given in (\ref{6}).

Since $Q\notin \textnormal{var}^{G}(A)$, for all $Q\in  \mathcal{I}$, then, by Lemma \ref{306}, for each $i = 1,\ldots , m$ we have $$B_i\sim_{T_{(G,*)}} \displaystyle {\bigoplus_{g\in G} }P_i^g,$$ where $P_i^g=F +(J_{11})_1^+ +J(B_i)_g$ and $P_i^h\in \textnormal{var}^{(G,*)}(M_{h,\rho})$, for all $h\in G-\{1\}$. As long as $Q\notin \textnormal{var}^{G}(A)$, for all $Q\in \mathcal{K}$, by \cite[Theorem 37]{lamattinamisso},  $P_i^1$ is $T_{(G,*)}$-equivalent to one of the following $(G,*)$-algebras:  $$N,\, C\oplus N, \,C_{2,*}^{1}\oplus N,\, A_{2}^*\oplus N\mbox{ or } C_{2,*}^{1}\oplus A_{2}^* \oplus N$$ where $N$ is a nilpotent $(G,*)$-algebra with trivial $G$-grading and $C$ is a commutative algebra with trivial involution and trivial $G$-grading.

Now, we fix $h\in G-\{1\}$ and recall that $N_{2}^\mathsf{h}\sim_{T_{(G,*)}} C_{2,*}^{h}$ and $U_{2}^\mathsf{h}\sim_{T_{(G,*)}} C_2^{h}$. Since $A_3^{\mathsf{h}},N_3^{\mathsf{h}}, U_3^{\mathsf{h}}\notin \textnormal{var}^{(G,*)}(A)$ then, by Theorem  \ref{31}, $P_i^h$ is $T_{(G,*)}$-equivalent to one of the following $(G,*)$-algebras:

\begin{center}
$N$, $C\oplus N$, $C_{2,*}^{h} \oplus N$,  $C_{2}^{h} \oplus N$, $A_{2}^{\mathsf{h}} \oplus N$, \\  $C_{2,*}^{h} \oplus C_2^{h} \oplus N$, $C_{2,*}^{h} \oplus A_{2}^{\mathsf{h}} \oplus N$, $C_2^{h} \oplus A_{2}^{\mathsf{h}} \oplus N$, $C_{2,*}^{h} \oplus C_2^{h} \oplus A_{2}^{\mathsf{h}} \oplus N.$
\end{center}

Finally, for all $g_j\in G$, we define $M^{g_j}= \displaystyle \bigoplus_{1\leq i \leq m} P_i^{g_j}$. Recalling that $A=B_1 \oplus \cdots \oplus B_m$ with either $B_i$ is a nilpotent $(G,*)$-algebra or $B_i\sim_{T_{(G,*)}} \displaystyle {\bigoplus_{g\in G} }P_i^g$ as above then $A$ is $T_{(G,*)}$-equivalent to $M^{g_1} \oplus \cdots \oplus  M^{g_t}\oplus N$, where $N$ is a nilpotent $(G,*)$-algebra. Therefore, the proof follows.
\end{proof}

As a consequence of the previous theorem, we can present the classification of minimal $(G,*)$-varieties of quadratic growth, extending Theorem \ref{minimaislinear}.

\begin{corolario} \label{minimais}
     Let $A$ be a finite-dimensional $(G,*)$-algebra of quadratic growth over a field $F$ of characteristic zero. Then $A$ generates a minimal variety if and only if $A\sim_{T_{(G,*)}} Q$, for some $Q\in \mathcal{M}$.
\end{corolario}

\begin{proof}
    Assume that $\textnormal{var}^{(G,*)}(A)$ is a minimal $(G,*)$-variety of quadratic growth. By Theorem \ref{1234}, there exists $Q\in \mathcal{M}$ such that $Q\in \textnormal{var}^G(A)$ and $Q$ has quadratic growth, by Remark \ref{min}. Thus, since $A$ generates a minimal $(G,*)$-variety then we have $\textnormal{var}^{(G,*)}(A)=\textnormal{var}^{(G,*)}(Q)$. 
    
  Now, suppose that $A\sim_{T_{(G,*)}} Q$, for some $Q\in \mathcal{M}$, and consider $\mathcal{U}\subsetneq \textnormal{var}^{(G,*)}(A)$ a proper subvariety. By Remark \ref{min} we have that $\textnormal{Id}^G(A)\not\subset \textnormal{Id}^G(B)$, for all $B\in \mathcal{M}$ with $B\neq Q$, and so $\textnormal{Id}^G(\mathcal{U})\not\subset \textnormal{Id}^G(S)$, for all $S\in \mathcal{M}$. By Theorem \ref{1234}, this implies that $\mathcal{U}$ is a variety of at most linear growth and then we are done.
\end{proof}

In the meantime we consider $G\cong \mathbb{Z}_2$ generated by an element $g$ with $g^2=1$ and define the following set of $*$-superalgebras: 
\begin{center}
    $\mathcal{L}=\{M_{4,\rho}^{u}, M_{5,\rho}^{u},$ $M_{i,\rho}^{u,v},  M_{j, \rho}^{p,q}, M_{6,\omega_1}^{u,v}, M_{7,\omega_2}^{u,v}$, $M_{11,\rho}^{g,u}$,  $ C_{3,*}^u, C_{3}^g, \mathcal{G}_{2, \diamond }^{g,g}, \mathcal{G}_{2, \imath}^{1,g},\mathcal{G}_{2, \tau}^{1,1}, W_{\nu}^{g,g}, W_{\eta}^{1,g}, U_3^*, N_3^*, A_3^\mathsf{g}, N_3^\mathsf{g}, U_3^\mathsf{g} \}-\{M_{6,\omega_1}^{1,1}\}$,
\end{center} for all $i=8,9,10$; $u,v,p,q\in G$, $p\neq q$;
$j=6,7$; $\diamond \in  \{\tau, \gamma, \psi\}$, $\imath \in \{\tau, \gamma\}$, $\nu\in \{\nu_1, \nu_2, \nu_3\}$, $\eta \in \{\nu_2, \nu_3\}$.

In conclusion, we have the following corollary which corresponds to Corollary $7.2$ of \cite{mallu}.

\begin{corolario}
    The $*$-superalgebras in $\mathcal{L}$ generate the only minimal $(\mathbb{Z}_2, *)$-varieties of quadratic growth.
\end{corolario}

 \end{document}